\documentclass{amsart}
\usepackage{hyperref,amsmath,amsfonts,amssymb,amsthm,enumerate,mathtools,graphicx,xparse,xcolor,tikz,scalerel,stackengine}
\usepackage[a4paper, left=2cm, right=2cm, top=3cm, bottom=2cm]{geometry}
\usepackage[utf8]{inputenc}
\usepackage{mathtools}
\usepackage{enumitem}
\usepackage{float}

\numberwithin{dummy}{section}
\numberwithin{equation}{section}

\theoremstyle{plain}
\newtheorem{thm}{Theorem}[section]
\newtheorem{lemma}[thm]{Lemma}

\theoremstyle{plain}
\newtheorem{df}[thm]{Definition}

\theoremstyle{remark}

\newcommand{\R}{\mathbb{R}}
\newcommand{\N}{\mathbb{N}}

\newcommand*\dd{\mathop{}\!\mathrm{d}}

\newcommand*{\ddt}{\mathop{}\!\frac{\mathrm{d}}{\mathrm{d}t}}
\def\di{\mathop{\mathrm{div}}\nolimits}

\NewDocumentCommand{\oldnorm}{sO{}m}{%
  {\IfBooleanTF{#1}
    {\oldnormaux{\left|}{\right|}{#3}}
    {\oldnormaux{#2|}{#2|}{#3}}}
}
\makeatletter
\newcommand{\oldnormaux}[3]{\mathpalette\oldnormaux@i{{#1}{#2}{#3}}}
\newcommand{\oldnormaux@i}[2]{\oldnormaux@ii#1#2}
\newcommand{\oldnormaux@ii}[4]{%
  \sbox\z@{$\m@th#1#2#4#3$}%
  \sbox\tw@{$\m@th\|$}%
  \mathopen{\hbox to\wd\tw@{\hss\vrule height \ht\z@ depth \dp\z@ width .3\wd\tw@\hss}}%
  #4
  \mathclose{\hbox to\wd\tw@{\hss\vrule height \ht\z@ depth \dp\z@ width .3\wd\tw@\hss}}%
}
\makeatother

\newcommand{\vect}[1]{\boldsymbol{#1}}

\newcommand{\abs}[1]{\left|{#1}\right|}

\newcommand{\Abs}[1]{\left\lVert{#1}\right\rVert}

\def\bq{\vect{q}}
\def\bg{\vect{g}}
\def\bz{\vect{z}}

\newcommand\pig[1]{\scalerel*[5.5pt]{\Big#1}{%
  \ensurestackMath{\addstackgap[1.5pt]{\big#1}}}}
\newcommand\pigl[1]{\mathopen{\pig{#1}}}
\newcommand\pigr[1]{\mathclose{\pig{#1}}}

\tikzstyle{dot}=[draw,fill=white,circle,inner sep=0pt,minimum size=4pt]
\tikzstyle{dot}=[draw,fill=white,circle,inner sep=0pt,minimum size=4pt]
\def\deltazero{
\begin{tikzpicture}[scale=2]
    \footnotesize
    \draw[ultra thin] (-0.2,  0.0) -- (2.2, 0.0) node[below]{$\abs{\nabla u}$};
    \draw[ultra thin] ( 0.0, -0.2) -- (0.0, 2.2) node[left ]{$|\bq|$};
      \draw[variable=\t,domain=0:2,thick,samples=400]
      plot ({\t}         ,{1}) node[draw,ultra thin,solid,right,xshift=3,yshift=-3] {$p=1$} ;
      \node[label={left:1}] at (0, 1) {};
      \node[label={below:1}] at (1, 0) {};
        \draw[variable=\t,domain=0:2,thick,samples=400]
      plot ({\t}         ,{(\t)^(1/2)} ) node[draw,ultra thin,right,xshift=3,yshift=3] {$p=\frac{3}{2}$};
    \draw[variable=\t,domain=0:2,thick,samples=400]
      plot ({\t}         ,{\t}         ) node[draw,ultra thin,above right,xshift=1.6,yshift=1.6] {$p=2$};
    \draw[variable=\t,domain=0:2,thick,samples=400]
      plot ({(\t)^(1/2)} ,{\t}         ) node[draw,ultra thin,above,yshift=3,xshift=8] {$p=3$};
    \draw[variable=\t,domain=0:2,thick,samples=400]
      plot (1,{\t}         ) node[draw,ultra thin,solid,above,yshift=3,xshift=-8] {$p=+\infty$};
    \draw[variable=\t,domain=0:1,thick,samples=400]
      plot (0,{\t}         );
      \draw[variable=\t,domain=0:1,thick,samples=400]
      plot ({\t},0         );
  \end{tikzpicture}
}

\def\deltaonea{
  \begin{tikzpicture}[scale=1.5]
    \footnotesize
    \draw[ultra thin] (-0.2,  0.0) -- (2.2, 0.0) node[below]{$|\nabla u|$};
    \draw[ultra thin] ( 0.0, -0.2) -- (0.0, 2.2) node[left ]{$|\bq|$};
    \node[label={left:1}] at (0, 1) {};
    \draw[variable=\t,domain=0:2,thin, dashed,samples=400]
      plot ({\t}         ,{1});
    \draw[variable=\t,domain=0:2,thick,samples=400,every node/.style={inner sep=1,outer sep=1}]
      plot ({\t},{(1+(\t)^(2))^(-1/2)*\t}) node[draw,ultra thin,right,xshift=3,yshift=-4] {$p=1$}
      plot ({\t},{(1+(\t)^(2))^(-1/4)*\t}) node[draw,ultra thin,right,xshift=3,yshift= 0.0] {$p=\frac 32$}
      plot ({\t},{\t}) node[draw,ultra thin,right,xshift=3,yshift= 0.0] {$p=2$};
      \draw[variable=\t,domain=0:1.3,thick,samples=400,every node/.style={inner sep=1,outer sep=1}]
      plot ({\t},{(1+(\t)^(2))^(1/2)*\t}) node[draw,ultra thin,right,xshift=3,yshift= 0] {$p=3$};
      \draw[variable=\t,domain=0:0.61,thick,samples=400,every node/.style={inner sep=1,outer sep=1}]
      plot ({\t},{(1+(\t)^(2))^(4)*\t}) node[draw,ultra thin,above,xshift=0,yshift=2] {$p=10$}
            ;
   \end{tikzpicture}
}

\def\deltaoneb{
  \begin{tikzpicture}[scale=1.5]
    \footnotesize
    \draw[ultra thin] (-0.2,  0.0) -- (2.2, 0.0) node[below]{$|\nabla u|$};
    \draw[ultra thin] ( 0.0, -0.2) -- (0.0, 2.2) node[left ]{$|\bq|$};
    \node[label={below:1}] at ( 1,0) {};
       \draw[variable=\t,domain=0:2,thin, dashed,samples=400]
      plot ({1},{\t});
    \draw[variable=\t,domain=0:2,thick,samples=400,every node/.style={inner sep=1,outer sep=1}]
      plot ({(1+(\t)^(2))^(-1/2)*\t},{\t}) node[draw,ultra thin,above,xshift=2,yshift= 2] {$p'=1$}
      plot ({(1+(\t)^(2))^(-1/4)*\t},{\t}) node[draw,ultra thin,right,xshift=1,yshift= 8.0] {$p'=\frac32$}
      plot ({\t},{\t}) node[draw,ultra thin,right,xshift=3,yshift= 0.0] {$p'=2$};
      \draw[variable=\t,domain=0:1.3,thick,samples=400,every node/.style={inner sep=1,outer sep=1}]
      plot ({(1+(\t)^(2))^(1/2)*\t},{\t}) node[draw,ultra thin,right,xshift=3,yshift= 0] {$p'=3$};
      \draw[variable=\t,domain=0:0.6,thick,samples=400,every node/.style={inner sep=1,outer sep=1}]
      plot ({(1+(\t)^(2))^(4)*\t},{\t}) node[draw,ultra thin,right,xshift=3,yshift=0.0] {$p'=10$}
            ;
  \end{tikzpicture}
}

\def\limiting{

\begin{tikzpicture}[scale=1.5]
    \footnotesize
    \draw[ultra thin] (-0.2,  0.0) -- (2.2, 0.0) node[below]{$|\bq|$};
    \draw[ultra thin] ( 0.0, -0.2) -- (0.0, 2.2) node[left ]{$|\nabla u|$};
    \draw[variable=\t,domain=0:2,thin, dashed,samples=400]
      plot ({\t}         ,{1});
    \node[label={left:$1$}] at (0, 1) {};
    \draw[variable=\t,domain=0:2,thick,samples=400,every node/.style={inner sep=1,outer sep=1}]
      plot ({\t},{(1+(\t)^(1/2))^(-2)*\t}) node[draw,ultra thin,right,xshift=3,yshift= -2.0] {$a=\frac12$}
      plot ({\t},{(1+(\t)^(1/1))^(-1)*\t}) node[draw,ultra thin,right,xshift=3,yshift= -2.0] {$a=1$}
      plot ({\t},{(1+(\t)^(2))^(-1/2)*\t}) node[draw,ultra thin,right,xshift=3,yshift=-2.5] {$a=2$}
      plot ({\t},{(1+(\t)^(6))^(-1/6)*\t}) node[draw,ultra thin,right,xshift=3,yshift= 2.5] {$a=6$}
      ;
    \draw[thick,variable=\t,every node/.style={inner sep=1,outer sep=1}]
      plot[domain=0:1] ({\t},{\t}) node[draw,ultra thin,above,xshift=4,yshift=4] {$a=+\infty$}
      --plot[domain=1:2] ({\t},{1});
 \end{tikzpicture}
 }

\begin{document}

\title{On evolutionary problems with a-priori bounded gradients}\thanks{M.~Bul\'{\i}\v{c}ek's work is supported by the project 20-11027X financed by Czech science foundation (GA\v{C}R). J.~M\'{a}lek acknowledges the support of the project No.\ 18-12719S financed by Czech Science Foundation (GA\v{C}R). M.~Bul\'{\i}\v{c}ek and J.~M\'{a}lek are  members of the Ne\v{c}as Center for Mathematical Modeling. The PhD position of D. Hru\v{s}ka is funded by the German Science Foundation DFG in context of the Priority Program SPP 2026 ``Geometry at Infinity''.}
\author[M.~Bul\'{i}\v{c}ek]{Miroslav Bul\'i\v{c}ek}
\address{Charles University, Faculty of Mathematics and Physics, Mathematical Institute\\
Sokolovsk\'{a} 83, 186 75 Prague, Czech Republic}
\email{mbul8060@karlin.mff.cuni.cz}

\author[D. Hru\v{s}ka]{David Hru\v{s}ka}
\address{Leipzig University, Faculty of Mathematics and Computer Science,
Institute of Mathematics,
Augustusplatz 10, 04 109 Leipzig, Germany}
\email{hruska@math.uni-leipzig.de}

\author[J. M\'{a}lek]{Josef M\'{a}lek}
\address{Charles University, Faculty of Mathematics and Physics, Mathematical Institute \\
Sokolovsk\'{a} 83, 186 75 Prague, Czech Republic}
\email{malek@karlin.mff.cuni.cz}

\begin{abstract}
    We study a nonlinear evolutionary partial differential equation that can be viewed as a generalization of the heat equation where the temperature gradient is a~priori bounded but the heat flux provides merely \mbox{$L^1$-coercivity}. Applying higher differentiability techniques in space and time, choosing a special weighted norm (equivalent to the Euclidean norm in $\mathbb{R}^d$), incorporating finer properties of integrable functions and using the concept of renormalized solution, we prove long-time and large-data existence and uniqueness of weak solution, with an $L^1$-integrable flux, to an initial spatially-periodic problem for all values of a positive model parameter. If this parameter is smaller than $2/(d+1)$, where $d$ denotes the spatial dimension, we obtain higher integrability of the flux. As the developed approach is not restricted to a scalar equation, we also present an analogous result for nonlinear parabolic systems in which the nonlinearity, being the gradient of a strictly convex function, gives an a-priori  $L^\infty$-bound on the gradient of the unknown solution.
\end{abstract}

\maketitle
\section{Introduction}
\subsection{Problem setting and main result}
This paper concerns a parabolic-like problem involving nonlinear elliptic operators that can be viewed as regularizations of the $\infty$-Laplacian. More precisely, for fixed $L>0$ and $T>0$ we set $\Omega:=(0,L)^d\subset\R^d$ and $Q:=(0,T)\times\Omega$ and investigate the following problem: for given \mbox{$\Omega$-periodic} functions $g:[0,T]\times \R^d\to\R$, $u_0:\R^d\to\R$ and a given parameter~$a>0$, find an $\Omega$-periodic function $u:[0,T]\times \R^d\to\R$ and a vectorial $\Omega$-periodic function $\vect{q}:[0,T]\times \R^d\to \R^d$ such that
\begin{subequations}\label{zadani_classic}
\begin{align}
    \label{rovnice_classic}
    \partial_t u-\di  \vect{q}  &= g &&\textrm{in }Q,\\
    \label{constit_classic}
    \nabla u&=\frac{\vect{q}}{(1+\abs{\vect{q}}^a)^{\frac 1a}} &&\textrm{in }Q,\\
    \label{initial_classic}
u(0, \cdot)&=u_0&& \textrm{in }\Omega.
\end{align}
\end{subequations}
The motivation for investigating such type of problems is given below. The main result of this paper 
is the following: \emph{for sufficiently smooth initial data $u_0$, which satisfies a reasonable compatibility condition, and for sufficiently smooth right-hand side $g$, there exists a unique couple $(u,\bq)$ solving~\eqref{zadani_classic} in the sense of distributions.} To formulate the result  precisely, we need to fix the notation, the appropriate function spaces and the concept of solution to~\eqref{zadani_classic}. Since we are dealing with a spatially periodic problem, we recall the definition of periodic Sobolev spaces
\begin{equation*}\label{periodic_def}
W_{per}^{k,p}(\Omega):=\overline{\left\{u=\tilde{u}_{\big|\Omega},\,\tilde{u}\in C^{\infty}(\R^d)\text{~is~}\Omega\text{-periodic} \right\}}^{\Vert\cdot\Vert_{k,p}},
\end{equation*}
where $k\in \mathbb{N}_0$ and $p\in[1,\infty)$ are arbitrary (note that $L^2_{per}(\Omega)=L^2(\Omega)$ and that these spaces, as closed subspaces of reflexive Banach spaces, are reflexive as well provided that $p\in (1,\infty)$). The space $W^{k,\infty}_{per}$ is then defined as
$$
W^{k,\infty}_{per}(\Omega):=W^{k,2}_{per}(\Omega)\cap W^{k,\infty}(\Omega).
$$
Throughout the paper, we use standard notation for Lebesgue, Sobolev and Bochner spaces equipped with the usual norms. Unless stated otherwise, bold letters, e.g. $\bq$, are used for vector-valued functions to distinguish them from scalar functions. The symbol ``$\partial_t$" stands for the partial derivative with respect to the time variable $t\in(0,T)$, while the operators ``$\nabla$" and  ``$\di$" take into account only the spatial variables $(x_1,\ldots, x_d)\in \Omega$. Later, we also use ``$\partial_j$'' to abbreviate partial derivative with respect to $x_j$.  The shortcut ``a.e." abbreviates \emph{almost everywhere} and ``a.a." stands for \emph{almost all}.

Next, we define the notion of a weak solution to~\eqref{zadani_classic} and formulate the main result.
\begin{df}\label{weak_formulation}
Let $u_0\in L^2(\Omega)$, $g\in L^2(Q)$ and $a>0$. We say that  a couple $(u,\vect{q})$ is a weak solution to problem~\eqref{zadani_classic} if
$$
\begin{aligned}
u&\in W^{1,2}\left(0,T; L^2(\Omega)\right)\cap L^2\left(0,T; W^{1,2}_{per}(\Omega)\right),\\
\bq&\in L^1\left(0,T; L^1\left(\Omega; \mathbb{R}^d\right)\right)
\end{aligned}
$$
and
\begin{subequations}\label{flux-integrable_weak_formulation}
\begin{align}
       \int_{\Omega} \partial_t u \,\varphi + \bq \cdot \nabla \varphi \dd x&=\int_{\Omega} g \, \varphi\dd x &&\textrm{for all $\varphi\in W^{1,\infty}_{per}(\Omega)$ and a.a. $t\in (0,T)$},\label{rovnice}
    \\
    \nabla u&=\frac{\vect{q}}{(1+\abs{\vect{q}}^a)^{\frac 1a}} &&\textrm{a.e. in }Q,\label{constit_old}
    \\
     \Vert u(t,\cdot)-u_0\Vert_{L^2(\Omega)}&\xrightarrow{t\to 0^+}0. \label{initial}
\end{align}
\end{subequations}
\end{df}
\begin{thm}\label{main_thm}
Let  $a>0$, $g\in L^2\left(0,T; L^{2}(\Omega)\right)$ and $u_0\in W_{per}^{1,\infty}(\Omega)$ satisfy
\begin{equation}\label{flux_small_data}
    \|\nabla u_0\|_{L^{\infty}(\Omega)}=:U<1.
\end{equation}
{\bf (i)} Then there exists a unique weak solution to problem~\eqref{zadani_classic} in the sense of Definition \ref{weak_formulation}. Moreover, the solution satisfies
\begin{equation}\label{regularityM}
u\in L^2\left(0,T;W^{2,2}_{per}(\Omega)\right).
\end{equation}
{\bf (ii)}
Furthermore, if 
$g\in W^{1,2}\left(0,T; L^2(\Omega)\right)$ and $u_0\in W_{per}^{2,2}(\Omega)$, then $u\in W^{1,\infty}(0,T; L^2(\Omega))$. If, in addition, the parameter $a$ satisfies
\begin{equation}\label{flux_small_a}
a\in \left(0,\frac{2}{d+1}\right),
\end{equation}
then
\begin{equation}\label{pepa6}
\vect{q}\in L^b(Q;\mathbb{R}^d) \quad \textrm{ for } \quad \,\, \begin{cases} b = \frac{(1-a)(d+1)}{d-1} > 1 &\textrm{if } d\ge 2, \\
b \textrm{ arbitrary} &\textrm{if } d=1.
\end{cases}
\end{equation}
\end{thm}
The paper is structured in the following way. In the rest of this section, we describe the main novelties of our result in detail. We also add a physical motivation for studying such problems and show the key difficulties of the studied problem. Section~\ref{preliminaries_section} contains several auxiliary results needed in the proof of Theorem~\ref{main_thm}. In Section~\ref{uniqueness_section}, we prove the uniqueness result. Sections~\ref{galerkin_section} and \ref{sec:5} concern the existence result. In Section~\ref{galerkin_section}, we introduce a suitable $\varepsilon$-approximation of the problem~\eqref{zadani_classic},  which is then treated by the standard Faedo-Galerkin method in combination with a cascade of energy estimates that helps to establish the existence of a weak solution to the $\varepsilon$-approximation for arbitrary fixed $\varepsilon\in (0,1)$.
Finally, we derive and summarize the whole cascade of estimates that are uniform with respect to $\varepsilon$. Then, in Section \ref{sec:5}, letting $\varepsilon\to 0+$, we incorporate the \emph{renormalization technique} together with a special choice of weigthed scalar product (equivalent to the standard scalar product in $\R^d$) to identify a weak solution of the original problem. Section~\ref{higher_integrability_section} is devoted to the proof of higher regularity (integrability) of the flux~$\vect{q}$ for the values of~$a$ satisfying~\eqref{flux_small_a}, which concludes the proof of the second part of Theorem \ref{main_thm}. In the final section, we formulate a generalization of the results stated in Theorem~\ref{main_thm}. 

\subsection{State of the art and main novelties}\label{background_section}

In order to put our result in an appropriate context, we introduce nonlinear (quasilinear) elliptic and parabolic problems characterized by the presence of $p$-Laplacian or its generalizations of various forms. Thus, for $d\in\N$, $a>0$, $\delta\in\{0,1\}$ and $p$ satisfying $1 < p \le \infty$, we define $\vect{f}_{\!p'}: \R^d\to \R^d$ by
\begin{equation}\label{formula_general}
    \vect{f}_{\!p'}(\vect{q}):=(\delta+\abs{\vect{q}}^a)^{\frac{p'-2}{a}}\vect{q}, \quad \textrm{ where } \,\, p' = \begin{cases} \frac{p}{p-1} &\textrm{if } p\in (1, \infty), \\ 1 &\textrm{if } p=\infty. \end{cases}
\end{equation}
Similarly, now for $p$ satisfying $1\le p < \infty$, we set $\vect{g}_p:\R^d \to \R^d$ as
\begin{equation}\label{formula_general2}
    \vect{g}_{p}(\vect{z}):=(\delta+\abs{\vect{z}}^a)^{\frac{p-2}{a}}\vect{z}.
\end{equation}
Replacing the equation \eqref{constit_classic} by
\begin{equation}\label{1.8}
\nabla u = \vect{f}_{\!p'}(\bq) \textrm{ with } \vect{f}_{\!p'} \textrm{ introduced in  \eqref{formula_general}},
\end{equation}
we obtain
\begin{equation}\label{pepa1}
\begin{aligned}
    \partial_t u-\di  \vect{q}  &= g &&\textrm{in }Q,\\
		\nabla u &= \left(\delta + |\vect{q}|^a\right)^{\frac{p'-2}{a}} \vect{q} &&\textrm{in }Q,\\
    u(0,\cdot)&=u_0&& \textrm{in }\Omega,
\end{aligned}
\end{equation}
while replacing \eqref{constit_classic} by
\begin{equation}\label{1.9}
\bq = \vect{g}_{p}(\nabla u) \textrm{ with } \bg_p \textrm{ introduced in \eqref{formula_general2}},
\end{equation}
we end up with
\begin{equation}
    \begin{aligned}
    \partial_t u-\di \left( (\delta + |\nabla u|^a)^{\frac{p-2}{a}} \nabla u \right) &= g &&\textrm{in }Q, \\
		u(0,\cdot)&=u_0&& \textrm{in }\Omega.
		\end{aligned}\label{pepa2}
\end{equation}

Next, let us first restrict ourselves to the case $p\in (1,\infty)$. Then,  the mappings $\vect{f}_{\!p'}$ and $\bg_{p}$ are strictly monotone for all $a>0$ and $\delta\in\{0,1\}$. In addition, when
$\delta=0$, $\vect{f}_{\!p'}=(\bg_{p})^{-1}$ and \eqref{pepa1} and \eqref{pepa2} coincide. Note that when $\delta=1$ the $(\bq,\nabla u)$-relations are smoothed out near zero (thus eliminating the degeneracy/singularity of the corresponding elliptic operator) and the problems \eqref{pepa1} and \eqref{pepa2} do not describe the same $(\bq,\nabla u)$-relation anymore. In all these cases the natural function spaces for the solution are as follows:
$$
\begin{aligned}
u&\in L^p\left(0,T; W^{1,p}_{per}(\Omega)\right)\cap W^{1,p'}\left(0,T;W^{1,p}_{per}(\Omega)^*\right),\\
\bq  &\in L^{p'}\left(0,T; L^{p'}\left(\Omega;\mathbb{R}^d\right)\right),
\end{aligned}
$$
provided that the data satisfy $u_0\in L^2_{per}(\Omega)$ and $g\in L^{p'}(0,T; W^{1,p}_{per}\left(\Omega)^*\right)$. Within this functional setting, the existence and uniqueness theory for such problems is nowadays classical, see~\cite{LaSoUr68,Li69} including and extending the monotone operator theory invented by Minty for the elliptic setting in Hilbert spaces (see~\cite{Minty}). It turns out that one can develop a rather complete theory for such problems and we refer to the classical monograph \cite{DiBe93} for additional regularity results. Furthermore, one can introduce a much more general class of possible relationships between $\bq$ and $\nabla u$ that goes far beyond \eqref{1.8} or \eqref{1.9} and where $\bq$ and $\nabla u$ are related implicitly. This means that instead of \eqref{constit_classic} one considers the equation $\vect{g}(\vect{q}, \nabla u) = \vect{0}$ in $Q$ with $\vect{g}:\R^{d}\times\R^d \to \R^d$ continuous. Under suitable assumptions imposed on $\vect{g}$, providing among others $p$-coercivity for $\nabla u$ and $p'$-coercivity for $\vect{q}$, a self-contained large-data mathematical theory within the above functional setting has been recently developed, also for the systems, in~\cite{BuMaMa20} (including, but also extending the results established in~\cite{BGMS1,BGMS2} in the context of fluid mechanics).

A natural and interesting question is what happens when $p\to 1^+$ or $p\to \infty$. In the case $\delta=0$, we formally obtain from~\eqref{1.9} for~$p=1$ that
$$
\bq = \frac{\nabla u}{|\nabla u|}.
$$
Then, the governing equation for the time-independent (stationary) problem being of the form \mbox{$-\di(\nabla u/|\nabla u|) = g$} formally represents the Euler-Lagrange equation corresponding to the minimization of the total variation functional. Analogously, and again for $\delta=0$, it follows from \eqref{1.8} that for $p=\infty$ (i.e. $p'=1$) one has
$$
\nabla u = \frac{\bq}{|\bq|},
$$
which,  together with the governing equation $-\di\vect{q} = g$, corresponds to the so-called $\infty$-Laplacian, see also Fig. \ref{Fig1}.
\begin{figure}[h]
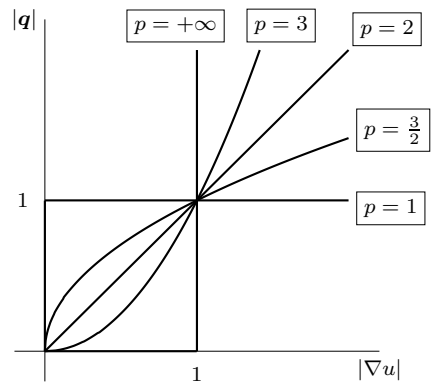

    \centering
    \deltazero{}
    \caption{If $p\in(1,\infty)$, then $\bq=\abs{\nabla u}^{p-2}\nabla u\Leftrightarrow\nabla u=\abs{\bq}^{p'-2}\bq$ with $p'=\frac{p}{p-1}$. Selected graphs are drawn (for values $p=\frac32,2,3$). The limiting cases $p=1$ and $p=\infty$ (i.e. $p'=1$) are sketched as well.}
    \label{Fig1}
\end{figure}

Both  limiting cases have attracted attention in the scientific community. Not only is  the understanding of these limiting cases interesting as a mathematical problem \emph{per se}, but also the total variation equation or $\infty$-Laplacian are frequently used when studying sharp interface-like problems, image recovering, etc. Let us point out that, in the elliptic (i.e. stationary) setting, one faces serious difficulties with defining a proper concept of solution and usually one has to introduce a new one. While for $p=1$ this has led to the theory of BV spaces, see e.g.~\cite{Gi84}, for $p=\infty$ the concept of \emph{viscosity} solution was introduced in~\cite{Ar67}. In principle, one can say that the expected $L^1$-regularity for $\nabla u$ (when $p=1$) or the $L^1$-regularity for $\bq$ (when $p=\infty$) must be relaxed and one is led to work in the ``weak$^{\ast}$ closure of $L^1$" or, more precisely, in the space of Radon measures. In the parabolic setting, there is a certain mollification effect coming from the presence of the time derivative and therefore the case $p=1$ is not so difficult to treat provided that the initial data are sufficiently regular, see e.g.~\cite{AnCaMa04}. However, for $p=\infty$, one  seems to be forced to keep the notion of a viscosity solution, see~\cite{AkJuKa09,PoVa13}. Furthermore, it is also well known that the viscosity solution is in principle the best object one can deal with, which is well documented by the existence of a singular solution (see~\cite{Ar86} or the monograph~\cite{Li16}).

The above discussion was focused on the case $\delta=0$, which leads to certain singular behaviour near zero. For a mollified problem with $\delta=1$, the limiting cases take the form
\begin{align*}
\bq &= \frac{\nabla u}{(1+|\nabla u|^a)^{\frac{1}{a}}} \quad\,\, \textrm{ for } p=1,\\
\nabla u &= \frac{\bq}{(1+|\bq|^a)^{\frac{1}{a}}} \qquad \textrm{ for } p=\infty,
\end{align*}
which may have better properties since both equations represent strictly monotone mapping unlike the case $\delta=0$, see also Fig. \ref{Fig2}. Nevertheless, even in this regularized case, one encounters difficulties. The most famous example concerns the case $a=2$ and $p=1$, i.e. the minimal surface problem. Due to Finn's counterexample (see~\cite{Fi65}), it is known that even for smooth data one can obtain an irregular solution that is not a Sobolev function. However, such a singularity appears only on (the Dirichlet part of) the boundary.  This follows from two results: the interior regularity established for the stationary problem with $p=1$ and $a\le2$ in \cite{BiFu02} and the existence result established in \cite{BeBuGm20} showing that the solution of the Neumann problem (for $p=1$ and $a>0$ arbitrary) is indeed a Sobolev function and there is no need to involve BV spaces. As this paper documents, a similar situation occurs the problems with $p=\infty$ and $\delta=1$.
\begin{figure}[h]
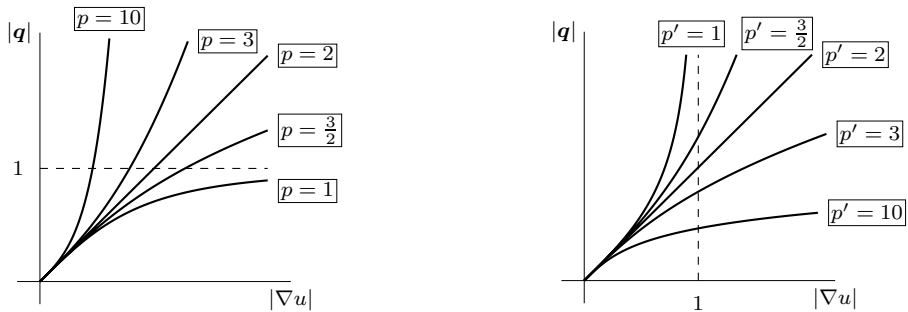

  \hfill\deltaonea\hfill\deltaoneb\hfill
  \caption {On the left, the graphs of
  $\bq=(1+\abs{\nabla u}^{2})^{\frac{p-2}{2}}\nabla u$ are sketched for selected values of $p\in[1,\infty)$, namely $p=1,\frac32,2,3,10$.
  On the right, the graphs of $\nabla u=(1+\abs{\bq}^{2})^{\frac{p'-2}{2}}\bq$ are shown for $p'=1,\frac{3}{2},2,3,10$.
  }
  \label{Fig2}
\end{figure}

Apparently, one could follow the procedure developed for $\infty$-Laplacian and try to treat the problem with the notion of viscosity solution. However, it is not clear how to adopt the theory of viscosity solution to our setting since we are dealing with a different elliptic operator (compare the limiting behaviour for $p=\infty$ and $\delta=0$ or $\delta=1$ depicted at Figures \ref{Fig1} and \ref{Fig2}). More importantly, it turns out (and this is one of the main messages of this paper) that we \emph{do not need} to introduce the concept of viscosity solution as we are able to establish the existence of a standard weak solution. Our method builds on the approach developed in~\cite{BuMaSu15} and~\cite{BeBuMaSu16}, where a similar elliptic problem arising in solid mechanics is analyzed. In this paper, we generalize the approach proposed in~\cite{BuMaSu15,BeBuMaSu16} (and used in some sense also in~\cite{BeBuGm20}) and adopt it to the parabolic setting.

An interesting problem might be the study of the limit $a\to \infty$. In such a case
\begin{equation*}
  (1+\abs{\vect{q}}^a)^{\frac{1}{a}}\searrow\max\{1,\abs{\vect q}\}\text{~as~}a\to\infty
\end{equation*}
and consequently (for $\vect{f}_{\!1}$ introduced in \eqref{formula_general})
\begin{equation*}
    \vect{f}_{\!1}(\vect{q})=\frac{\vect{q}}{(1+\abs{\vect{q}}^a)^{\frac 1a}}\nearrow \frac{\vect{q}}{\abs{\vect{q}}} \min\left\{1,\abs{\vect{q}}\right\}\text{~as~}a\to\infty.
\end{equation*}
However, the limiting mapping is not \emph{strictly} monotone (see Fig. \ref{Fig3}) and the method developed in this paper cannot be applied.
\begin{figure}[h]
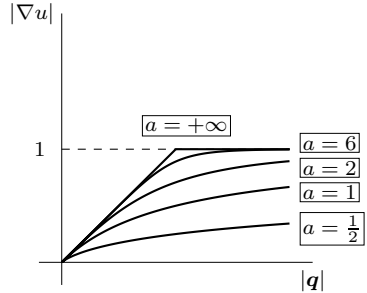

\centering
\limiting{}
\caption{The graphs of $\nabla u=\frac{\bq}{(1+\abs{\bq}^a)^{\frac{1}{a}}}$ are drawn for selected values of parameter $a\in(0,\infty)$. The limiting case $a=\infty$ is sketched as well.}
\label{Fig3}
\end{figure}

To summarize and emphasize the novelty of our result once again, we show the existence of a weak solution to the evolutionary problem~\eqref{zadani_classic} for all $a>0$ with no need to introduce the concept of viscosity solution and with $\bq$ being an integrable function.

It is worth mentioning that our proof of Theorem~\ref{main_thm}, as presented below, is based on two properties of the nonlinear function~$\vect{f}_{\!1}$ defined in~\eqref{formula_general}, namely, its radial structure, i.e. $\vect{f}_{\!1}(\bq) = \alpha(|\bq|)\bq$, and the existence of strictly convex potential to $\vect{f}_{\!1}$. Consequently, the specific form of the equation \eqref{constit_classic} is not essential and we can develop a satisfactory theory for a general class of relations behaving like mollified $\infty$-Laplacian (provided that there is a strictly convex potential behind). We state such a generalized result in Theorem~\ref{main_thm2} in Section~\ref{rem-final} but do not provide the proof for simplicity here.  However, an interested reader can compare our proof with the general methods invented in~\cite{BeBuMa18} for the elliptic setting. In fact, by adopting these methods and combining them with the proof of Theorem~\ref{main_thm}, one can prove Theorem~\ref{main_thm2}.

\subsection{A fluid mechanics problem motivating this study}\label{simple} Consider an incompressible fluid with constant density flowing, at a uniform temperature, in a three-dimensional domain. In the absence of external body forces, unsteady flows of such a fluid are described by the following set of equations for the unknown velocity field $\vect{v}= (v_1,v_2,v_3)$ and the pressure $p$:
\begin{equation}
\di\vect{v} = 0, \qquad \partial_t \vect{v} + \sum_{k=1}^3 v_k \partial_{k} \vect{v} = - \nabla p  + \di \mathbb{S}\,, \label{fluid1}
\end{equation}
where $\mathbb{S}$, the deviatoric part of the Cauchy stress tensor, enters the additional (so-called constitutive) equation relating $\mathbb{S}$ to the symmetric part of the velocity gradient denoted by $\mathbb{D}$ and characterizing the material properties of a particular class of fluids. While for the Newtonian fluids one has $\mathbb{S} = 2\nu_* \mathbb{D}$, where $\nu_*>0$ is the kinematic viscosity, there are many viscous fluids and fluid-like materials in which the relation between $\mathbb{S}$ and $\mathbb{D}$ is nonlinear. There are fluids (see for example~\cite{Lee2003,Ref1,Ref2,Ref3,Ref4}) in which the constitutive relation capable of describing experimental data can be of the form 
\begin{equation}
    \label{constit}
    2\nu_*\mathbb D=\frac{\mathbb S}{\left(1+\pigl(\frac{1}{\sqrt{2}}\abs{\mathbb S}\pigr)^a\right)^{\frac 1a}} \qquad \textrm{ for some } a>0 \textrm{ and } \nu_*>0.
\end{equation}
The general goal is to understand mathematical properties associated with the system of partial differential equations \eqref{fluid1}-\eqref{constit}. A possible natural approach is to look first at a geometrically simplified version of the problem. For example, one can investigate simple shear flows taking place between two infinite parallel plates located at $x_2=0$ and $x_2=L$. Time-dependent simple shear flows are characterized by the velocity field of the form $\vect{v}(t,x_1,x_2,x_3)=(u(t,x_2),0,0)$. Note that such velocity field fulfills $\di \vect{v}=0$. We also infer that the only nontrivial components of $\mathbb{D}$ are $\mathbb{D}_{12}=\mathbb{D}_{21}=\frac{1}{2} \partial_2 u$. Hence it follows from \eqref{constit} that also all components of $\mathbb S$ other than $\mathbb{S}_{12}=\mathbb{S}_{21}=:\sigma=\sigma(t,x_2)$ vanish. Then the second equation in \eqref{fluid1} together with \eqref{constit} leads to:
\begin{subequations}\label{system_basic}
\begin{align}
     \label{moment_x}\partial_t u&=-\partial_1 p + \partial_2 \sigma, \qquad 0 = -\partial_2 p, \qquad 0 = -\partial_3 p, \\
    \label{simple_constit}
    \nu_* \partial_2 u&=\frac{\sigma}{(1+\abs{\sigma}^a)^{\frac 1a}}.
\end{align}
\end{subequations}
It follows from the second and the third equation in \eqref{moment_x} that $p=p(t,x_1)$. After inserting this piece of information into the first equation of \eqref{moment_x} we can decompose this equation and obtain
\begin{equation}\label{pepa5}
(\partial_t u - \partial_2 \sigma)(t, x_2) = g(t) \qquad \textrm{ and } \qquad - \partial_1 p (t,x_1)=g(t)
\end{equation}
for some function $g$ depending only on time.
When studying the unsteady Poiseuille flow, the function $g$, corresponding to the pressure drop, must be given. Then the first equation in~\eqref{pepa5} together with~\eqref{simple_constit} represents a one-dimensional version of the governing equations of the problem~\eqref{zadani_classic} studied in this paper (with the caveat that in~\eqref{zadani_classic} the function~$g$ may also depend on the spatial variable).


\subsection{Difficulties and main idea}\label{difficulties}
As mentioned above, the key difficulty is due to a weak a~priori estimate for $\bq$ compensating the fact that $\nabla u$ is bounded a~priori. To be more explicit, let us recall the definition~\eqref{formula_general} with $\delta=1$, i.e. $\vect{f}_{\!1}(\vect{q}):= \frac{\vect{q}}{(1+ |\vect{q}|^a)^{\frac{1}{a}}}$. Obviously, $\abs{\vect{f}_{\!1}(\vect{q})}=\frac{\abs{\vect{q}}}{(1+\abs{\vect{q}}^a)^{\frac 1a}}<1$ for all $\vect{q}\in\R^d$. This directly yields that $\nabla{u}\in L^{\infty}(Q;\R^d)$, but it also brings the restriction that the inverse function of (the injective function) $\vect{f}_{\!1}$ cannot be defined outside of the unit ball in $\R^d$ and hence we may not simply write $\vect{q}$ as a function of $\nabla{u}$ and directly apply the Faedo--Galerkin approximation method.

Next, standard energy estimates are not sufficient to establish the existence of a weak solution. Indeed, multiplying the linear equation \eqref{rovnice_classic} by the solution $u$, integrating by parts with respect to the spatial variables (the spatial periodicity ensures that the boundary terms vanish) and substituting for $\nabla u$ from \eqref{constit_classic} we conclude that
$$
\int_{Q} \frac{\abs{\vect{q}}^2}{(1+\abs{\vect{q}}^a)^{\frac{1}{a}}} \dd x \dd t < \infty.
$$
However, this implies merely that $\vect{q}$ belongs to $L^1(Q;\R^d)$ which is not a reflexive Banach space (it does not even have a predual). Hence, when constructing a solution,  we may not identify a weak limit of a subsequence of $\left\{\vect{q}^n\right \}_{n=1}^{\infty}$, a sequence of some approximations bounded in $L^1(Q;\R^d)$. Similar difficulties occur if one aims to investigate the limiting behaviour when converging from the $p$-Laplacian to the $\infty$-Laplacian, i.e. when studying the limit $p'\to 1+$ in~\eqref{formula_general}.


At this point one might consider a~priori estimates involving higher derivatives. Let us denote by $s$ a general time or spatial variable, i.e. $s$ can represent $t, x_1, \ldots, x_d$. Let us differentiate the equation \eqref{rovnice_classic} with respect to $s$, multiply the result by $\partial_s u$ and integrate over $\Omega$. Finally, in the integral involving $\bq$, we integrate by parts and obtain
\begin{equation*}\label{formal_s_carkou}
    \frac 12  \ddt \Vert \partial_s u\Vert_{L^2(\Omega)}^2+\int_{\Omega}  \partial_s\vect{q}\cdot \partial_s(\nabla  u) \dd x=\int_{\Omega} \partial_s g \partial_s u\dd x.
\end{equation*}
Hence, if the data are sufficiently regular, one can hope for an a~priori estimate for $\vect{q}$ of the form
\begin{equation}\label{odhady_s_carkou}
    \int_{Q} \partial_s\vect{q}\cdot\partial_s(\nabla u)\dd x \dd t<\infty.
\end{equation}

Let us now focus on the information coming from \eqref{odhady_s_carkou} for general $\vect{f}_{\!p'}$ with $p'\in [1,\infty)$. Using \eqref{formula_general} (cf. Lemma \ref{derivace}) one obtains
\begin{equation}\label{vypocet_s_carkou_2}
    \partial_s\vect{q}\cdot \partial_s(\nabla u)=(1+\abs{\vect{q}}^a)^{\frac{p'-2-a}{a}}\left(\abs{\partial_s\vect{q}}^2(1+\abs{\vect{q}}^a)+(p'-2)\abs{\vect{q}}^{a-2}(\vect{q}\cdot \partial_s\vect{q})^2\right).
\end{equation}
For $p'>1$ we have $p'-2>-1$ and we can employ the Cauchy--Schwarz inequality for the last term to obtain the estimate
\begin{equation*}
     \partial_s\vect{q}\cdot \partial_s(\nabla u)\geq C(1+\abs{\vect{q}}^a)^{\frac{p'-2}{a}}\abs{\partial_s\vect{q}}^2,
\end{equation*}
where $C:=\min\{p'-1,1\}>0$  and this can be exploited to control $\partial_s\vect{q}$ in $L^s(Q;\R^d)$ for some $s>1$. However, in the critical case $p'=1$, there is a sudden loss of information as one  then deduces merely the estimate
\begin{equation}\label{o_jednicku}
     \partial_s\vect{q}\cdot\partial_s(\nabla u)\geq (1+\abs{\vect{q}}^a)^{\frac{-1-a}{a}}\abs{\partial_s\vect{q}}^2.
\end{equation}
Consequently, the power of $\abs{\vect{q}}$ in this weighted estimate drops by $a$. For small values of $a$, namely for those satisfying \eqref{flux_small_a}, it can be deduced from \eqref{odhady_s_carkou} and \eqref{o_jednicku} using Sobolev embedding that $\bq$ is bounded in $L^b(Q;\mathbb{R}^d)$ for some $b>1$, see \eqref{pepa6}. This is shown in the proof of the second part of Theorem~\ref{main_thm}. However, for large values of~$a$, the estimate \eqref{o_jednicku} seems to be useless at the first glance. We will however show that it implies \emph{almost everywhere convergence} for a selected subsequence of $\{\vect{q}^m\}$. This is still not sufficient to take the limit in the governing equation (due to $L^1$-integrability of $\{\vect{q}^m\}$). This is why we introduce the concept of \emph{renormalized solution} for a suitable $m$-approximating problem and then, in order to take the limit from the renormalized formulation of the approximate problem to the weak formulation of the original problem, we shall work directly with the quantity $\partial_s\vect{q}\cdot \partial_s(\nabla u)$ (or more precisely with the right-hand side of \eqref{vypocet_s_carkou_2}), which in some sense still generates an estimate for $\partial_s \bq$ in some scalar product in $\mathbb{R}^d$ induced by $\bq$ itself.
%


\section{Preliminaries}\label{preliminaries_section}

\noindent Here and in the remaining parts of this text we set, for $a>0$,
\begin{equation}
\vect{f}(\vect{q}):=\frac{\vect{q}}{(1+\abs{\vect{q}}^a)^{\frac 1a}} \quad \textrm{ where } \vect{q}\in \mathbb{R}^d.\label{dec:1}
\end{equation}
The aim of this section is to collect basic properties of $\vect{f}$ as well as its $\varepsilon$-approximation $\vect{f}^{\varepsilon}$ defined, for $\varepsilon >0$, as:
\begin{equation}\label{dec:2}
\vect{f}^{\varepsilon}(\vect{q}):=\vect{f}(\vect{q}) + \varepsilon\vect{q} =  \frac{\vect{q}}{(1+\abs{\vect{q}}^a)^{\frac 1a}}+\varepsilon\vect{q}.
\end{equation}
\begin{lemma}\label{derivace} The following assertions hold true:
\begin{enumerate}[label={(\roman*)}]
    \item\label{prvni_cast_item} $\vect{f}$, $\vect{f}^{\varepsilon}\in C^1(\R^d;\R^d)$ and for all $i,j=1,\ldots, d$ and arbitrary $\vect{q}\in \mathbb{R}^d$ there holds:
\begin{equation}\label{prvni_cast}
   \left(\nabla_{\!\vect{q}}\vect{f} (\vect{q})\right)_{ij}:= \frac{\partial f_i(\vect{q})}{\partial q_j} =\frac{(1+\abs{\vect{q}}^a)\delta_{ij}-\abs{\vect{q}}^{a-2}q_iq_j}{(1+\abs{\vect{q}}^a)^{1+\frac 1a}} \quad \textrm{ and } \quad \left(\nabla_{\!\vect{q}} \vect{f}^{\varepsilon} (\vect{q})\right)_{ij} = \left(\nabla_{\!\vect{q}}  \vect{f} (\vect{q})\right)_{ij} + \varepsilon\delta_{ij},
\end{equation}
where $\delta_{ij}$ is the Kronecker delta.
\item Introducing the scalar functions $f(s):=\frac{s}{(1+s^a)^{\frac 1a}}$ and $f_{\varepsilon}(s):= f(s) +\varepsilon s$ we have the following ``radial" representations for $\vect{f}$ and $\vect{f}^{\varepsilon}$:
\begin{equation}\label{prepis}
    \vect{f} (\vect{q})=f(\abs{\vect{q}})\frac{\vect{q}}{\abs{\vect{q}}} \quad \textrm{ and  }\quad \vect{f}^{\varepsilon}(\vect{q})=f_{\varepsilon}(\abs{\vect{q}})\frac{\vect{q}}{\abs{\vect{q}}} \quad \textrm{ for every } \vect{q}\neq\vect{0}.
\end{equation}
\item  For $\varepsilon> 0$ the function $\vect{f}^{\varepsilon}$ is a diffeomorphism from $\R^d$ onto $\R^d$, while $\vect{f}$ is a diffeomorphism from $\R^d$ onto the open unit ball $B_1(0)\subset\R^d$.
\end{enumerate}
\end{lemma}
\begin{proof} For $\vect{q}\neq\vect{0}$ 
we have
\begin{equation}\label{dukaz_lematka}\notag
   \frac{\partial f_i^{\varepsilon}(\vect{q})}{\partial q_j} = \frac{\partial}{\partial q_j}
   \left(\frac{q_i}{(1+\abs{\vect{q}}^a)^{\frac 1a}}\right)+\varepsilon\delta_{ij}=\frac{(1+\abs{\vect{q}}^a)\delta_{ij}-\abs{\vect{q}}^{a-2}q_iq_j}{(1+\abs{\vect{q}}^a)^{1+\frac 1a}}+\varepsilon\delta_{ij}.
\end{equation}
This result can be easily extended to $\vect{q}=\vect0$. Indeed, the above formula for partial derivatives is clearly continuous on $\R^d\setminus\{\vect 0\}$ and since $a>0$ and $\abs{q_iq_j}\leq\abs{\vect q}^2$ for all $i,j\in\{1\dots,d\}$, we conclude $\abs{\vect{q}}^{a-2}q_iq_j\rightarrow 0$ as $\vect{q}\to\vect 0$. Thus $\vect{f}, \vect{f}^{\varepsilon}\in C^1(\R^d;\R^d)$. This proves the first assertion.

As the vectors $\vect{q}$ and $\vect{f}^{\varepsilon}(\vect{q})$ have the same direction, the formulae \eqref{prepis} follow. Furthermore, $\lim_{s\to 0^+} f (s)=0$, $\lim_{s\to \infty}f (s)=1$ and $f'(s)=(1+s^a)^{-\frac{1+a}{a}} >0$. Consequently, $f$ is a strictly increasing $C^1$-function mapping $[0,\infty)$ onto $[0,1)$ and, for any $\varepsilon>0$, $f_{\varepsilon}$ is a strictly increasing $C^1$-function mapping $[0,\infty)$ onto $[0,\infty)$. Hence the functions
$$
\vect{f}^{-1}(\vect{y}):=f^{-1}\left(|\vect{y}|\right)\frac{\vect{y}}{\abs{\vect{y}}} \quad \textrm{ and } \quad
\left(\vect{f}^{\varepsilon}\right)^{-1}(\vect{y}):=\left(f_{\varepsilon}\right)^{-1}\left(|\vect{y}|\right)\frac{\vect{y}}{\abs{\vect{y}}}
$$
are well defined inverse functions of $\vect{f}$ and $\vect{f}^{\varepsilon}$, respectively. It is straightforward to check that $\vect{f}^{-1}$ and $\left(\vect{f}^{\varepsilon}\right)^{-1}$ are continuously differentiable, which completes the proof of (ii) and (iii).
\end{proof}
Next, we set
\begin{equation}\label{dec:3}
    \mathbb{A}(\vect{q}):= \nabla_{\!\vect{q}}\vect{f} (\vect{q}) \quad \textrm{i.e.} \quad \mathbb{A}(\vect{q}) = \frac{(1+\abs{\vect{q}}^a)\mathbb{I}-\abs{\vect{q}}^{a-2}\bq\otimes\bq}{(1+\abs{\vect{q}}^a)^{1+\frac 1a}}
\end{equation}
and we focus on its (finer) properties. (In \eqref{dec:3}, $\mathbb{I}$ stands for the identity matrix and $(\bq\otimes\bq)_{ij} = q_iq_j$.)
\begin{lemma}[Scalar product generated by $\nabla_{\!\vect{q}} \vect f(\vect{q})$]\label{new_scalar_product}
Let $\vect{q}\in\R^d$ be arbitrary. The bilinear form on $\R^d$ given by
\begin{equation}\label{new_scalar_product_formula}
    (\vect{v},\vect{w})_{\mathbb{A}(\vect{q})}:=\vect{v}\cdot \mathbb{A}(\vect{q})\vect{w}=\sum_{i,j=1}^d v_i \frac{\partial f_i(\vect{q})}{\partial q_j} w_j=\frac{(1+\abs{\vect{q}}^a)\vect{v}\cdot \vect{w}-\abs{\vect{q}}^{a-2}(\vect{q}\cdot \vect{v}) (\vect{q}\cdot \vect{w})}{(1+\abs{\vect{q}}^a)^{1+\frac 1a}}
\end{equation}
is a scalar product on $\mathbb{R}^d$ satisfying
\begin{equation}\label{nerovnost_B}
(\vect{v},\vect{w})_{\mathbb{A}(\vect{q})} \leq 2\abs{\vect v}\abs{\vect w} \quad\text{for every $\vect{v},\vect{w}\in\R^d$}.
\end{equation}
The corresponding quadratic form $\oldnorm{\vect{v}}^2_{\mathbb{A}(\vect{q})}:=(\vect{v},\vect{v})_{\mathbb{A}(\vect{q})}$ fulfills
\begin{equation}\label{odhady_s_carkou_new}
\abs{\vect{v}}^2 \geq \frac{ \abs{\vect{v}}^2}{(1+\abs{\vect{q}}^a)^{\frac 1a}} \geq\oldnorm{\vect{v}}^2_{\mathbb{A}(\vect{q})} \ge \frac{|\vect{v}|^2}{(1+\abs{\vect{q}}^a)^{1+\frac 1a}}\qquad\text{for every $\vect{v}\in\R^d$}
\end{equation}
Hence, $\oldnorm{\cdot}_{\mathbb{A}(\vect{q})}$ is for fixed $\bq\in\R^d$ the norm on $\mathbb{R}^d$ equivalent to the Euclidean norm $|\cdot|$.
\end{lemma}
\begin{proof}The proof follows from the definition of $\vect{f}$, the formula \eqref{prvni_cast} for its derivatives, \eqref{new_scalar_product_formula} and the Cauchy-Schwarz inequality. The inequalities in \eqref{odhady_s_carkou_new} are direct consequences of  \eqref{new_scalar_product_formula}.
\end{proof}
The last essential property we need in the proof is the strict monotonicity of $\vect{f}$, the strong monotonicity of $\vect{f}^{\varepsilon}$ and, consequently, the Lipschitz continuity of its inverse function $(\vect{f}^{\varepsilon})^{-1}$.
\begin{lemma}\label{monotonie}
The mappings $\vect{f}, \vect{f}^{\varepsilon}:\R^d\to\R^d$ defined in \eqref{dec:1} and \eqref{dec:2} satisfy, for all $\varepsilon\in(0,1)$,
\begin{align}
\bigl(\vect{f}(\vect{q}_1)-\vect{f}(\vect{q}_2)\bigr)\cdot(\vect{q}_1-\vect{q}_2) &>0 &&\textrm{for all } \vect{q}_1,\vect{q}_2\in\R^d, \vect{q}_1\neq \vect{q}_2, \label{dec:5}\\
\bigl(\vect{f}^{\varepsilon}(\vect{q}_1)-\vect{f}^{\varepsilon}(\vect{q}_2)\bigr)\cdot(\vect{q}_1-\vect{q}_2) &\ge \varepsilon |\vect{q}_1 - \vect{q}_2|^2 &&\textrm{for all } \vect{q}_1,\vect{q}_2\in\R^d. \label{dec:6}
\end{align}
Moreover, for any $\varepsilon>0$, the inverse function $\left(\vect{f}^{\varepsilon}\right)^{-1}$ is uniformly Lipschitz continuous on $\R^d$, namely,
\begin{align}\label{dec:19}
    \abs{(\vect{f}^{\varepsilon})^{-1}(\vect{y}_1)-(\vect{f}^{\varepsilon})^{-1}(\vect{y}_2)} &\le \frac{1}{\varepsilon} \abs{\vect{y}_1-\vect{y}_2} &&\textrm{for all } \vect{y}_1,\vect{y}_2\in\R^d.
\end{align}
\end{lemma}
\begin{proof}
We first observe, using also \eqref{dec:3}, that (for $\vect{q}_1\neq\vect{q}_2$)
\begin{align*}
    \bigl(\vect{f}^{\varepsilon}(\vect{q}_1)&-\vect{f}^{\varepsilon}(\vect{q}_2)\bigr) \cdot(\vect{q}_1-\vect{q}_2)= \int_0^1\frac{\mathrm{d}}{\mathrm{d}s}\vect{f}^{\varepsilon}(\vect{q}_2+s(\vect{q}_1-\vect{q}_2)) \dd s\cdot (\vect{q}_1-\vect{q}_2) \\ &=\int_0^1\mathbb{A}(\vect{q}_2+s(\vect{q}_1-\vect{q}_2)) (\vect{q}_1-\vect{q}_2) \cdot (\vect{q}_1-\vect{q}_2) \dd s + \varepsilon |\vect{q}_1-\vect{q}_2|^2
    > \varepsilon |\vect{q}_1-\vect{q}_2|^2,
 \end{align*}
which gives the strong monotonicity of $\vect{f}^{\varepsilon}$ and strict monotonicity of $\vect{f}$. Since
$$
\bigl(\vect{f}^{\varepsilon}(\vect{q}_1)-\vect{f}^{\varepsilon}(\vect{q}_2)\bigr) \cdot(\vect{q}_1-\vect{q}_2)
\le
\abs{\vect{f}^{\varepsilon}(\vect{q}_1)-\vect{f}^{\varepsilon}(\vect{q}_2)}\abs{\vect{q}_1-\vect{q}_2},$$
we conclude from the last two inequalities that $\varepsilon \abs{\vect{q}_1-\vect{q}_2} \le \abs{\vect{f}^{\varepsilon}(\vect{q}_1)-\vect{f}^{\varepsilon}(\vect{q}_2)}$, which is equivalent to \eqref{dec:19}.
\end{proof}

\section{Proof of uniqueness}\label{uniqueness_section}

\noindent In this short section, we shall prove that there is at most one weak solution to the problem \eqref{zadani_classic}.

Let us assume that there are two weak solutions $(u_1,\vect{q}_1)$ and $(u_2,\vect{q}_2)$ to the problem \eqref{zadani_classic} with the same initial value $u_0\in L^2(\Omega)$ and the same right-hand side $g\in L^2(Q)$. Note that the constitutive equation \eqref{constit_old} implies that $\nabla u_1, \nabla u_2 \in L^{\infty}(Q)$ and consequently $u_1$ and $u_2$ are admissible test function in \eqref{rovnice}. Subtracting \eqref{rovnice} for $(u_2,\vect{q}_2)$ from the same equation for $(u_1,\vect{q}_1)$ and taking $\varphi=u_1(t,\cdot)-u_2(t,\cdot)$ as a test function, we obtain
\begin{equation}\label{dec:4}
\int_{\Omega} (\partial_tu_1-\partial_tu_2)(u_1-u_2) + (\bq_1-\bq_2)\cdot \left(\nabla u_1-\nabla u_2\right) \dd x=0\qquad\text{for a.a.~}t\in(0,T).
\end{equation}
By \eqref{constit_old}, $\nabla u_1-\nabla u_2 = \vect{f}(\bq_1) - \vect{f}(\bq_2)$. Inserting this relation into \eqref{dec:4}, we obtain
$$
\frac12 \ddt \|u_1-u_2\|_{L^2(\Omega)}^2 + \int_{\Omega} (\vect{f}(\bq_1) - \vect{f}(\bq_2)) \cdot (\bq_1- \bq_2) \dd x = 0.
$$
Integrating this with respect to time $t\in (0, T]$ and using $u_1(0,x) - u_2(0,x) = 0$ a.e. in $\Omega$ we arrive at
$$
\frac12 \|u_1(t, \cdot)-u_2(t, \cdot)\|_{L^2(\Omega)}^2 + \int_0^t \int_{\Omega} \left(\vect{f}(\bq_1) - \vect{f}(\bq_2) \right) \cdot  (\bq_1- \bq_2) \dd x \dd s = 0.
$$
By taking $t=T$ and using the strict monotonicity of $\vect{f}$, see \eqref{dec:5}, the second term leads to the conclusion that $\bq_1=\bq_2$ a.e. in $(0,T)\times\Omega$. The first term then implies that, for all $t\in (0,T]$, $u_1(t,\cdot) = u_2(t,\cdot)$ a.e. in $\Omega$.
This completes the proof of uniqueness.

\section{\texorpdfstring{$\varepsilon$}{e}-approximations and their properties}\label{galerkin_section}

\noindent In this section, we introduce, for any $\varepsilon\in (0,1)$, an $\varepsilon$-approximation of the problem \eqref{zadani_classic} and show, by means of the Galerkin method and regularity techniques performed at the Galerkin level, that this  $\varepsilon$-approximation admits a unique weak solution with second spatial derivatives in $L^2(Q)$.

Let $\varepsilon\in (0,1)$ and $a>0$. We say that a couple of $\Omega$-periodic functions $(u,\bq)=(u^{\varepsilon},\bq^{\varepsilon})$ solves the  \mbox{$\varepsilon$-approximation} of the problem \eqref{zadani_classic} if
\begin{subequations}\label{zadani_eps}
\begin{align}
    \label{rovnice_eps}
    \partial_t u-\di  \vect{q}  &= g &&\textrm{in }Q,\\
    \label{constit_eps}
    \nabla u&=\frac{\vect{q}}{(1+\abs{\vect{q}}^a)^{\frac 1a}} + \varepsilon \bq = \vect{f}(\bq) + \varepsilon \bq = \vect{f}^{\varepsilon}(\bq)&&\textrm{in }Q,\\
    \label{initial_eps}
u(0, \cdot)&=u_0&& \textrm{in }\Omega.
\end{align}
\end{subequations}

In accordance with the assumptions of Theorem \ref{main_thm}, we assume that $u_0 \in W^{1,\infty}_{per}(\Omega)$ satisfies \eqref{flux_small_data} and $g\in L^2(Q)$. We say that a couple $(u,\bq) = (u^{\varepsilon},\bq^{\varepsilon})$ is \emph{weak solution to \eqref{zadani_eps}} if
\begin{equation}\label{dec:34}
\begin{aligned}
u&\in L^2\left(0,T; W^{2,2}_{per}(\Omega)\right),\\
\partial_t u&\in L^2\left(0,T; L^2(\Omega)\right), \\
\bq&\in L^2\left(0,T; L^2\left(\Omega; \mathbb{R}^d\right)\right)
\end{aligned}
\end{equation}
and
\begin{subequations}\label{wf-eps}
\begin{align}
       \int_{\Omega} \partial_t u \,\varphi + \bq \cdot \nabla \varphi \dd x&=\int_{\Omega} g \, \varphi\dd x &&\textrm{for all $\varphi\in W^{1,2}_{per}(\Omega)$ and a.a. $t\in (0,T)$},\label{wf-rovnice}
    \\
    \nabla u&=\vect{f}^{\varepsilon}(\vect{q}) &&\textrm{a.e. in }Q,\label{wf-ce}
    \\
     \Vert u(t,\cdot)-u_0\Vert_{L^2(\Omega)}&\xrightarrow{t\to 0^+}0. \label{wf-in}
\end{align}
\end{subequations}
Uniqueness of such a solution follows from the same argument as in Section \ref{uniqueness_section}. To establish the existence of the solution, we apply the Galerkin method combined with higher differentiability estimates that we will perform at the level of Galerkin approximations.
These estimates and the limit from the Galerkin approximation to the continuous level represent the core of this section. In Subsect. \ref{apriori_revisited}, we establish and summarize the estimates that are uniform with respect to $\varepsilon$.

\subsection{Galerkin approximations}\label{Galerkin} Consider the basis $\left\{\omega_r\right\}_{r=1}^{\infty}$ in $W_{per}^{1,2}(\Omega)$ consisting of solutions
of the following spectral problem:
\begin{align}\label{Dec:GS}
    \int_{\Omega}\nabla\omega_r\cdot\nabla\varphi\dd x&=\lambda_r\int_{\Omega}\omega_r\varphi\dd x \qquad\text{~for all~} \varphi\in W^{1,2}_{per}(\Omega).
\end{align}
It is well-known (see e.g. \cite{temam} or \cite[Appendix A.4]{mnrr96}) that there is a non-decreasing sequence of (positive) eigenvalues $\{\lambda_r\}_{r=1}^{\infty}$ and a corresponding set of eigenfunctions $\left\{\omega_r\right\}_{r=1}^{\infty}$ that are orthogonal in $W^{1,2}_{per}(\Omega)$ and orthonormal in $L^2_{per}(\Omega)$. Moreover, the projections  $\mathcal{P}^N$ defined through $\mathcal{P}^N(u)=\sum_{i=1}^N \left(\int_{\Omega} u \omega_i\dd x\right)\omega_i$ are continuous both as mappings from $L^2_{per}(\Omega)$ to $L^2_{per}(\Omega)$ and from $W^{1,2}_{per}(\Omega)$ to $W^{1,2}_{per}(\Omega)$.
Also, due to $\Omega$-periodicity and elliptic regularity, the $\Omega$-periodic extensions of $\omega_r$ belong to $C^{\infty}(\R^d)$.

Before introducing the Galerkin approximations of the problem \eqref{wf-eps} we recall, referring to Lemma \ref{derivace}, that the relation $\nabla u = \vect{f}^{\varepsilon}(\bq)$ is equivalent to $\bq = (\vect{f}^{\varepsilon})^{-1}(\nabla u)$ where $(\vect{f}^{\varepsilon})^{-1}$ is a Lipschitz mapping from $\mathbb{R}^d$ to $\mathbb{R}^d$.

For an arbitrary, fixed $N\in \N$ , we look for $u^N$ in the form
\begin{equation}\label{galerkin_def}\notag
    u^N(t,x)=\sum_{r=1}^Nc_r^N(t)\, \omega_r(x),
\end{equation}
where the coefficients  $c_r^N$, $r=1, \dots, N$, are determined as the solution of the system of ordinary differential equations of the form
\begin{subequations}\label{galerkin_0}
\begin{align}\label{galerkin1}
   \int_{\Omega}\partial_t{u}^N \omega_r + \vect{q}^N \cdot \nabla\omega_r\dd x&=\int_{\Omega}g\,\omega_r\dd x, \quad r=1,\ldots, N, \qquad \textrm{ where } \quad \vect{q}^N := (\vect{f}^{\varepsilon})^{-1}(\nabla u^N),  \\
   \label{galerkin2}
    u^N(0, \cdot)&= \mathcal{P}^N(u_0) \qquad \iff \qquad c^N_r(0)= \int_{\Omega} u_0 \omega_r \dd x \qquad r=1,\dots,N.
\end{align}
\end{subequations}
The local-in-time well-posedness of the above problem \eqref{galerkin_0} directly follows from Caratheodory theory (recall here that $\left(\vect{f}^{\varepsilon}\right)^{-1}$ is a Lipschitz mapping). In addition, thanks to the first uniform estimates established in the next subsection, we deduce that the Galerkin system~\eqref{galerkin_0} is well-posed on $(0,T]$.

\subsection{First uniform estimates}\label{apriori_estimates}
Multiplying the $r$-th equation in~\eqref{galerkin1} by $c_r$ and summing these equations up for $r=1,\dots,N$, we obtain
\begin{equation}\label{apriori1_dif}\notag
    \frac 12 \ddt \Abs{ u^N}_{L^2(\Omega)}^2+\int_{\Omega}\vect{q}^N \cdot \nabla u^N\dd x=\int_{\Omega}  g\, u^N\dd x.
\end{equation}
Using the one-to-one correspondence between $\vect{q}^N$ and $\nabla u^N$, see \eqref{galerkin1}, the second term on the left-hand side can be evaluated explicitly and the above equation takes the form
\begin{equation}\label{apriori1_difA}\notag
    \frac 12 \ddt \Abs{ u^N}_{L^2(\Omega)}^2+\int_{\Omega}\frac{\abs{\vect{q}^N}^2}{\left(1+\abs{\vect{q}^N}^a\right)^{\frac 1a}} +\varepsilon \left|\vect{q}^N\right|^2 \dd x=\int_{\Omega}  g\, u^N\dd x \le \frac12 \|g\|_{L^2(\Omega)}^2 + \frac12 \left\|u^N\right\|_{L^2(\Omega)}^2.
\end{equation}
Integrating over time, using then the Gronwall inequality and the fact that  $\|\mathcal{P}^N u_0\|_{L^2(\Omega)}\le \|u_0\|_{L^2(\Omega)}$, we obtain
 \begin{equation}\label{uniform}
  \sup_{t\in(0,T)} \Abs{u^N(t,\cdot)}_{L^2(\Omega)}^2+ \int_0^T\int_{\Omega}\frac{\abs{\vect{q}^N}^2}{\left(1+\abs{\vect{q}^N}^a\right)^{\frac 1a}} +\varepsilon \left|\vect{q}^N\right|^2 \dd x\dd t\leq \mathcal{C}\!\left(\Abs{ u_0}_{L^2(\Omega)},\Abs{ g}_{L^2(Q)}\right).\footnote{By symbols such as $\mathcal{C}\!\left(\Abs{ u_0}_{L^2(\Omega)},\Abs{ g}_{L^2(Q)}\right)$ we indicate the dependence of the finite upper bound on ``relevant" parameters (i.e. $u_0$, $g$, $a$ and auxiliary parameters introduced in the proof such as $\varepsilon$). The value of this bound can change from line to line.}
\end{equation}
In addition, it also directly follows from $\nabla u^N = \vect{f}^{\varepsilon}(\bq^N)$ (see the second equation in \eqref{galerkin1}) and the above $L^2$ estimate on $\bq^N$ that
\begin{equation}\label{uniform-b}
  \int_0^T\int_{\Omega}\left|\nabla u^N\right|^2 \dd x\dd t\leq \mathcal{C}\!\left(\Abs{ u_0}_{L^2(\Omega)},\Abs{ g}_{L^2(Q)}\right).
\end{equation}

\subsection{Time derivative estimate (uniform with respect to \texorpdfstring{$N$}{N})}\label{first_time_der}
Multiplying the $r$-th equation in~\eqref{galerkin1} by $\ddt c_r$ and summing these equations up for $r=1,\dots,N$, we obtain
\begin{equation*}
   \int_{\Omega} \left|\partial_t u^N\right|^2 +\vect{q}^N \cdot \partial_t \left(\nabla u^N\right)\dd x=\int_{\Omega}  g\, \partial_t u^N\dd x.
\end{equation*}
Applying Young's inequality to the term on the right-hand side, we get
\begin{equation}\label{apriori1_tt1}
   \int_{\Omega} \left|\partial_t u^N\right|^2 +2\vect{q}^N \cdot \partial_t \left(\nabla u^N\right)\dd x\le \int_{\Omega}  |g|^2 \dd x.
\end{equation}
Next, we focus on the second term on the left-hand side. Since $\nabla u^N = \vect{f}^{\varepsilon}(\bq^N)$, it follows from the definition of $\vect{f}^{\varepsilon}$ that
\begin{align*}
\vect{q}^N \cdot \partial_t \left(\nabla u^N\right)&= \partial_t \left( \vect{q}^N \cdot \nabla u^N\right)- \partial_t \vect{q}^N \cdot \nabla u^N\\
&=\partial_t \left(\frac{\abs{\vect{q}^N}^2}{\left(1+\abs{\vect{q}^N}^a\right)^{\frac 1a}} +\varepsilon \left|\vect{q}^N\right|^2 \right)- \partial_t \vect{q}^N \cdot \left(\frac{\vect{q}^N}{\left(1+\abs{\vect{q}^N}^a\right)^{\frac 1a}} +\varepsilon \vect{q}^N \right)\\
&=\frac{\varepsilon}{2} \partial_t\left( \abs{\bq^N}^2\right)  + \partial_t \left(\frac{\abs{\vect{q}^N}^2}{\left(1+\abs{\vect{q}^N}^a\right)^{\frac 1a}}  \right)- \partial_t \left(\abs{\vect{q}^N}\right)  \frac{\abs{\bq^N}}{\left(1+\abs{\vect{q}^N}^a\right)^{\frac 1a}} \\
&=\frac{\varepsilon}{2} \partial_t \left(\abs{\bq^N}^2\right) + \partial_t \int_0^{\abs{\bq^N}} \left(\frac{\abs{\vect{q}^N}}{\left(1+\abs{\vect{q}^N}^a\right)^{\frac 1a}}  -\frac{s}{(1+s^a)^{\frac 1a}}  \right)\dd s.
\end{align*}
Inserting the result of this computation into  \eqref{apriori1_tt1}, integrating the result over  $(0,T)$, and using the fact that the function
$$
s\mapsto \frac{s}{(1+s^a)^{\frac 1a}}
$$
is increasing (implying that $\abs{\vect{q}^N}(1+\abs{\vect{q}^N}^a)^{-\frac 1a}-{s}{(1+s^a)^{-\frac 1a}} \ge 0$ on $(0, \abs{\vect{q}^N}$), we obtain that \begin{align*}
\int_0^T\! &\int_{\Omega} \abs{\partial_t u^N}^2 \dd x \dd t \le \int_0^T\!\int_{\Omega} |g|^2 - 2\bq^N \cdot \partial_t\left(\nabla u^N\right)\dd x \dd t\\ &=  \int_0^T\!\int_{\Omega} |g|^2 \dd x \dd t -\left[\int_{\Omega}\varepsilon\abs{\bq^N(t,x)}^2+2\!\int_0^{\abs{\bq^N(t,x)}} \!\!\left(\frac{\abs{\vect{q}^N(t,x)}}{\left(1+\abs{\vect{q}^N(t,x)}^a\right)^{\frac 1a}}  -\frac{s}{(1+s^a)^{\frac 1a}}  \right)\!\dd s\dd x\right]_{t=0}^{t=T}
\\&\le \int_0^T\!\int_{\Omega} |g|^2 \dd x \dd t+ \int_{\Omega}\varepsilon\abs{\bq^N(0,x)}^2 + 2\int_0^{\abs{\bq^N(0,x)}} \!\!\left(\frac{\abs{\vect{q}^N(0,x)}}{\left(1+\abs{\vect{q}^N(0,x)}^a\right)^{\frac 1a}}  -\frac{s}{(1+s^a)^{\frac 1a}} \dd s\right)\dd x. 
\end{align*}
Noticing that $\abs{\vect{q}^N}(1+\abs{\vect{q}^N}^a)^{-\frac 1a}-{s}{(1+s^a)^{-\frac 1a}} \le 1$ on $(0, \abs{\vect{q}^N})$
we conclude that
\begin{equation}\label{fitme-f}
\begin{split}
&\int_0^T \int_{\Omega} \abs{\partial_t u^N}^2 \dd x \dd t \le
\Abs{g}^2_{L^2(Q)} +\varepsilon \Abs{ \bq^{N}(0,\cdot)}_{L^2(\Omega;\mathbb{R}^d)}^2+2\Abs{\bq^{N}(0,\cdot)}_{L^1(\Omega;\mathbb{R}^d)},
\end{split}
\end{equation}
where
\begin{equation}\label{dec:16}
   \bq^N(0,x) = \left(\vect{f}^{\varepsilon}\right)^{-1}(\nabla \mathcal{P}^N (u_0(x))) \quad \iff \quad \nabla \mathcal{P}^N (u_0) = \frac{\vect{q}^N(0,\cdot)}{\left(1+\abs{\vect{q}^N(0,\cdot)}^a\right)^{\frac 1a}} + \varepsilon \vect{q}^N(0,\cdot).
\end{equation}
Consequently,
$$
  |\vect{q}^N(0,\cdot)| \le \frac{1}{\varepsilon} |\nabla \mathcal{P}^N (u_0)|,
$$
which implies that
$$
\|\bq^N(0,\cdot)\|_{L^1(\Omega;\mathbb{R}^d)} \le |\Omega|^{1/2} \|\bq^N(0,\cdot)\|^{1/2}_{L^2(\Omega;\mathbb{R}^d)} \le \frac{1}{\varepsilon} |\Omega|^{1/2} \|\nabla \mathcal{P}^N(u_0)\|^{1/2}_{L^2(\Omega;\mathbb{R}^d)}.
$$
The fact that $\|\mathcal{P}^N (u_0)\|_{W^{1,2}_{per}(\Omega)} \le \|u_0\|_{W^{1,2}_{per}(\Omega)}$ thus finally yields
\begin{equation}\label{fitme-fb}
\begin{split}
&\int_0^T \int_{\Omega} \abs{\partial_t u^N}^2 \dd x \dd t \le \mathcal{C}\!\left(\varepsilon^{-1}, \|g\|_{L^2(Q)}, \|u_0\|_{W^{1,2}_{per}(\Omega)}\right).
\end{split}
\end{equation}

\subsection{Spatial derivative estimates}\label{space_der}~This time, we multiply the $r$th equation in \eqref{galerkin1} by $\lambda_r c_r$ and sum the obtained identities up for $r=1,\dots,N$. Since, due to \eqref{Dec:GS} and the smoothness of $\omega^r$,
$$
\lambda_r \int_{\Omega} \omega_r \varphi \dd x = \int_{\Omega}\nabla\omega_r\cdot\nabla\varphi\dd x =- \int_{\Omega}\Delta \omega_r\varphi\dd x \qquad\text{~for all~} \varphi\in W^{1,2}_{per}(\Omega),
$$
we get
\begin{align}\notag
    \int_{\Omega}\partial_t \nabla u^N \cdot \nabla u^N + \nabla \bq^N\cdot \nabla^2 u^N \dd x&=-\int_{\Omega} g \, \Delta u^N \dd x.
\end{align}
Hence,
\begin{equation}\label{rovnost_s_gradientem11}
  \ddt \Abs{\nabla u^N}_{L^2(\Omega;\mathbb{R}^d)}^2  + 2\int_{\Omega}\nabla \bq^N \cdot \nabla^2 u^N \dd x=-2\int_{\Omega}  g \, \Delta u^N \dd x\le 2\|g\|_{L^2(\Omega)}\Abs{\nabla^2 u^N}_{L^2(\Omega; \mathbb{R}^{d\times d})}.
\end{equation}
Since $\nabla u^N = \vect{f}^{\varepsilon}(\vect{q}^N)$, recalling \eqref{dec:3} we get
$$
  \nabla ^2 u^N = \mathbb{A}(\vect{q}^N) \nabla \vect{q}^N + \varepsilon \nabla \vect{q}^N.
$$
Hence, by Lemma \ref{new_scalar_product}, we get
\begin{equation}\label{dec:7}
\nabla \bq^N \cdot \nabla ^2 u^N = \nabla \bq^N \cdot \mathbb{A}(\bq^N)\nabla \bq^N + \varepsilon |\nabla \bq^N|^2 = \oldnorm{\nabla \bq^N}^2_{\mathbb{A}(\vect{q}^N)} + \varepsilon |\nabla \bq^N|^2
\end{equation}
and also, by means of the Cauchy-Schwarz inequality  and \eqref{odhady_s_carkou_new},
\begin{align*}
\abs{\nabla^2 u^N}^2 &= \mathbb{A}(\vect{q}^N) \nabla \bq^N \cdot \nabla^2 u^N + \varepsilon \nabla \bq^N \cdot \nabla^2 u^N \le \oldnorm{\nabla \bq^N}_{\mathbb{A}(\vect{q}^N)} \oldnorm{\nabla ^2 u^N}_{\mathbb{A}(\vect{q}^N)} + \varepsilon |\nabla \bq^N|\, |\nabla ^2 u^N| \\
& \le \oldnorm{\nabla \bq^N}_{\mathbb{A}(\vect{q}^N)} |\nabla ^2 u^N| + \varepsilon |\nabla \bq^N|\, |\nabla ^2 u^N|,
\end{align*}
which, using $\varepsilon^2 < \varepsilon$, implies that
\begin{equation}\label{dec:8}
  |\nabla^2 u^N|^2 \le 2( \oldnorm{\nabla \bq^N}^2_{\mathbb{A}(\vect{q}^N)} + \varepsilon |\nabla \bq^N|^2).
\end{equation}
Incorporating \eqref{dec:7} and \eqref{dec:8} into \eqref{rovnost_s_gradientem11}, integrating the result with respect to time and using Young's inequality and the continuity of $\mathcal{P}^N$ in $W^{1,2}_{per}(\Omega)$, we arrive at estimates that are uniform with respect to both $N$ and $\varepsilon$:
\begin{align}\label{odhady_gradient}
\begin{split}
\sup_{t\in(0,T)}\Abs{ \nabla u^N(t,\cdot)}_{L^2(\Omega;\mathbb{R}^d)}^2+ \int_0^T \int_{\Omega}\oldnorm{\nabla \bq^N}^2_{\mathbb{A}(\vect{q}^N)} +\varepsilon \abs{\nabla \bq^N}^2+\abs{\nabla^2 u^N}^2\dd x \dd t\\
\leq \mathcal{C}\!\left(\Abs{g}_{L^2(Q)}\!,\Abs{u_0}_{W^{1,2}_{per}(\Omega)}\right).
\end{split}
\end{align}
\subsection{Limit \texorpdfstring{$N\to \infty$}{Ntoinfty}}
Due to the reflexivity and separability of the underlying function spaces and the Aubin-Lions compactness lemma, it follows from the estimates \eqref{uniform}, \eqref{uniform-b}, \eqref{fitme-fb} and \eqref{odhady_gradient} that there is a subsequence of $\left\{(u^N, \vect{q}^N)\right\}_{N=1}^{\infty}$ (which we do not relabel) such that 
\begin{subequations}\label{limity_N}
\begin{align}
    \label{u_slabe_N}
    u^N&\rightharpoonup u &&\text{~weakly in~} L^2\left(0,T; W^{2,2}_{per}(\Omega)\right),\\ \label{casova_slabe_N}
    \partial_t{u}^N&\rightharpoonup \, \partial_t {u} &&\text{~weakly in~} L^{2}\left(0,T;L^{2}(\Omega)\right), \\ \label{strong_u}
    u^N&\to u &&\text{~strongly in~} L^2\left(0,T; W^{1,2}_{per}(\Omega)\right) \cap C\left([0,T];L^2(\Omega)\right), \\
    \label{q_slabe_N}  \vect{q}^N&\rightharpoonup \vect{q} &&\text{~weakly in~} 
    L^2\pigl(0,T;W_{per}^{1,2}\left(\Omega;\mathbb{R}^d\right)\pigr).
\end{align}
\end{subequations} 
Letting $N\to \infty$ in \eqref{galerkin_0}, it is simple to conclude from the above convergence results that
\begin{align}
\begin{split}\label{limita_rce_N_elegant}
    \int_{\Omega}\partial_t u \, \varphi+\vect{q}\cdot \nabla\varphi \dd x&=\int_{\Omega} g \, \varphi\dd x \qquad  \textrm{for all } \varphi\in W^{1,2}_{per}(\Omega) \textrm{ and a.a. } t\in (0,T].
    \end{split}
\end{align}
Since $u^N(0, \cdot)=\mathcal{P}^N(u_0)$, $\mathcal{P}^N(u_0) \to u_0$ in $L^2(\Omega)$ and $u\in C\left([0,T];L^2(\Omega)\right)$, we observe that \eqref{wf-in} holds.

By virtue of \eqref{strong_u} there  is a subsequence (that we again do not relabel) so that
\begin{equation}\label{dec:18}
    \nabla u^N\xrightarrow{N\to\infty}\nabla u\quad\text{a.e. in } Q.
\end{equation}
As $\left(\vect{f}^{\varepsilon}\right)^{-1}$ is (Lipschitz) continuous, it follows from the second equation in \eqref{galerkin1} and \eqref{dec:18} that
\begin{equation*}
    \vect{q}^N = \left(\vect{f}^{\varepsilon}\right)^{-1}\left(\nabla u^N\right) \xrightarrow{N\to\infty}  \left(\vect{f}^{\varepsilon}\right)^{-1}\left(\nabla u\right) \quad\text{a.e. in } Q.
\end{equation*} Since the weak limit in $L^2(Q)$ coincides with the pointwise limit a.e. in $Q$ (provided that these limits exist), we conclude that
\begin{equation}\label{constit_N_done}
    \left(\vect{f}^{\varepsilon}\right)^{-1}\left(\nabla u\right)=\bq\quad \text{a.e. in~}Q\quad\Longrightarrow \quad\nabla u=\vect{f}^{\varepsilon}(\vect{q}) \quad \text{a.e. in~}Q.
\end{equation}
Thus, the existence and uniqueness of a weak solution to the $\varepsilon$-approximation \eqref{zadani_eps} in the sense of definition \eqref{wf-eps} is completed.

In the next subsection, we establish and summarize the estimates associated with the $\varepsilon$-approximation \eqref{zadani_eps} that are uniform with respect to $\varepsilon$.


\subsection{\texorpdfstring{$\varepsilon$}{}-independent estimates for \texorpdfstring{$(u^{\varepsilon}, \vect{q}^{\varepsilon})$}{u,q}} \label{apriori_revisited} Observing that $u^{\varepsilon}$ is an admissible test function in \eqref{limita_rce_N_elegant}, we set $\varphi = u^{\varepsilon}$ in \eqref{limita_rce_N_elegant}. Then, proceeding step by step as at the Galerkin level, we obtain
\begin{equation}\label{dec:17}
  \sup_{t\in(0,T)} \Abs{u^{\varepsilon}(t,\cdot)}_{L^2(\Omega)}^2+ \int_0^T\int_{\Omega}\frac{\abs{\vect{q}^{\varepsilon}}^2}{\left(1+\abs{\vect{q}^{\varepsilon}}^a\right)^{\frac 1a}} +\varepsilon \left|\vect{q}^{\varepsilon}\right|^2 \dd x\dd t\leq \mathcal{C}\!\left(\Abs{ u_0}_2,\Abs{ g}_{L^2(Q)}\right).
\end{equation}
It is easy to conclude from the boundedness of the second term, by applying H\"{o}lder's inequality, that
\begin{equation}\label{dec:17a}
 \int_0^T\int_{\Omega}\abs{\vect{q}^{\varepsilon}} \dd x\dd t\leq \mathcal{C}\!\left(|\Omega|, \Abs{ u_0}_2,\Abs{ g}_{L^2(Q)}\right).
\end{equation}
Further estimates are obtained by taking the limit $N\to \infty$ in the estimates obtained at the Galerkin level.

We define $\bq(0,\cdot)$ through the equation
\begin{equation}\label{dec:15b}
\nabla u_0 = \vect{f}^{\varepsilon}(\bq(0,\cdot)) = \frac{\bq(0,\cdot)}{(1+ |\bq(0,\cdot)|^a)^{1/a}} + \varepsilon \bq(0,\cdot).
\end{equation}
As $\nabla \mathcal{P}^N(u_0) = \vect{f}^{\varepsilon}(\bq^N(0,\cdot))$, see \eqref{dec:16}, $\nabla \mathcal{P}^N(u_0) \to \nabla u_0$ strongly in $L^2(\Omega;\mathbb{R}^d)$, and $(\vect{f}^{\varepsilon})^{-1}$ is Lipschitz continuous, we conclude that
$$
  \vect{q}^N(0,\cdot) \xrightarrow{N\to\infty}\bq(0, \cdot) \quad \text{~strongly in~}L^2(\Omega;\mathbb{R}^{d}).
$$
Consequently, we can take the limit $N\to \infty$ in \eqref{fitme-f} and conclude, using also the weak lower semicontinuity of the $L^2$-norm together with \eqref{casova_slabe_N}, that
\begin{equation}\label{dec:15}
\begin{split}
&\int_0^T \int_{\Omega} \abs{\partial_t u^{\varepsilon}}^2 \dd x \dd t \le
\Abs{g}^2_{L^2(Q)} +\varepsilon \Abs{ \bq (0,\cdot)}_{L^2(\Omega;\mathbb{R}^d)}^2+2\Abs{\bq(0,\cdot)}_{L^1(\Omega;\mathbb{R}^d)}.
\end{split}
\end{equation}
It follows from  \eqref{flux_small_data} and \eqref{dec:15b} that
$$
U\ge |\nabla u_0| = \left(\frac{1}{(1+|\bq(0,\cdot)|^a)^{\frac{1}{a}}}+ \varepsilon\right) |\bq(0,\cdot)|\ge \frac{|\bq(0,\cdot)|}{(1+|\bq(0,\cdot)|^a)^{\frac{1}{a}}} \quad \textrm{ a.e. in } Q.
$$
This implies that
\begin{equation*}
|\bq(0,\cdot)|\le \frac{U}{(1-U^a)^{\frac{1}{a}}}.
\end{equation*}
As $U<1$ (see \eqref{flux_small_data}), we get
\begin{equation}
 \|\bq(0,\cdot)\|_{L^1(\Omega; \mathbb{R}^d)} \le \mathcal{C}(a, U) \quad \textrm{ and } \quad \|\bq(0,\cdot)\|_{L^2(\Omega; \mathbb{R}^d)} \le \mathcal{C}(a, U).\label{priprava2}
\end{equation}
The bound $\mathcal{C}(a,U) \to \infty$ as $a\to 0+$ or as $U\to 1-$. Inserting \eqref{priprava2} into \eqref{dec:15}, we get
\begin{equation}\label{dec:14}
\int_0^T \int_{\Omega} \abs{\partial_t u^{\varepsilon}}^2 \dd x \dd t \le \mathcal{C}(a,U,\Abs{g}_{L^2(Q)}).
\end{equation}

Finally, we let $N\to \infty$ in \eqref{odhady_gradient}. Recalling \eqref{q_slabe_N} and also \eqref{dec:18} together with \eqref{constit_N_done}, we have
\begin{align*}
\begin{aligned} 
     \nabla \bq^N &\rightharpoonup \nabla \bq \qquad &&\textrm{weakly in } L^2(Q; \mathbb{R}^{d\times d}),\\
     \bq^N &\to \bq \qquad &&\textrm{a.e. in } Q.
\end{aligned}
\end{align*}
This implies (see the next subsection for the proof in a slightly more general setting) that
$$
   \int_0^T \int_{\Omega}\oldnorm{\nabla \bq}^2_{\mathbb{A}(\vect{q})} \dd x \dd t \le \liminf_{N\to \infty} \int_0^T \int_{\Omega}\oldnorm{\nabla \bq^N}^2_{\mathbb{A}(\vect{q}^N)} \dd x \dd t.
$$
Consequently, letting $N\to \infty$ in \eqref{odhady_gradient}, we get
\begin{align}\label{dec:13b}
\begin{split}
\sup_{t\in(0,T)}\Abs{ \nabla u^{\varepsilon}(t,\cdot)}_{L^2(\Omega;\mathbb{R}^d)}^2+ \int_0^T \int_{\Omega}\oldnorm{\nabla \bq^{\varepsilon}}^2_{\mathbb{A}(\vect{q}^{\varepsilon})} +\varepsilon \abs{\nabla \bq^{\varepsilon}}^2+\abs{\nabla^2 u^{\varepsilon}}^2\dd x \dd t\\
\leq \mathcal{C}\!\left(\Abs{g}_{L^2(Q)}\!,\Abs{u_0}_{W^{1,2}_{per}(\Omega)}\right).
\end{split}
\end{align}

\subsection{Weak lower semicontinuity of the weighted \texorpdfstring{$L^2$}{L2}-norm}\label{dec:63} Here, we shall prove the following statement: if
\begin{align}
\label{dec:12a}
      \bz^n &\rightharpoonup \bz \qquad \textrm{weakly in } L^2(Q; \mathbb{R}^{d}) &\textrm{as } n\to \infty,\\
\label{dec:12b}
\bq^n &\to \bq \qquad \textrm{a.e. in } Q &\textrm{as } n\to\infty,
\end{align}
then
\begin{equation}
  \label{dec:11} \int_Q \oldnorm{\bz}^2_{\mathbb{A}(\vect{q})} \dd x \dd t \le \liminf_{n\to\infty} \int_Q \oldnorm{\bz^n}^2_{\mathbb{A}(\vect{q}^n)} \dd x \dd t.
\end{equation}
To prove it, we first recall that $\oldnorm{\bz}^2_{\mathbb{A}(\vect{q})} = \bz \cdot \mathbb{A}(\bq)\bz$, where $\mathbb{A}$ is introduced in \eqref{dec:3}. Observing that
$$
0 \le \oldnorm{\bz^n - \bz}^2_{\mathbb{A}(\vect{q}^n)} = \oldnorm{\bz^n}^2_{\mathbb{A}(\vect{q}^n)} - \oldnorm{\bz}^2_{\mathbb{A}(\vect{q}^n)} - 2 (\bz, \bz^n - \bz)_{\mathbb{A}(\vect{q}^n)},
$$
we get
\begin{equation}
  \label{dec:10} \int_Q \oldnorm{\bz^n}^2_{\mathbb{A}(\vect{q}^n)} \dd x \dd t \ge \int_Q \oldnorm{\bz}^2_{\mathbb{A}(\vect{q}^n)} \dd x \dd t + 2 \int_Q (\bz, \bz^n - \bz)_{\mathbb{A}(\vect{q}^n)} \dd x \dd t.
\end{equation}
Since $|\mathbb{A}(\bq^n)|\le C(d)$ and \eqref{dec:12b} holds, Lebesgue's dominated convergence theorem implies that
\begin{equation}\label{dec:21}
   \lim_{n\to\infty} \int_Q \oldnorm{\bz}^2_{\mathbb{A}(\vect{q}^n)} \dd x \dd t = \lim_{n\to\infty} \int_Q \bz\cdot \mathbb{A}(\vect{q}^n) \bz \dd x \dd t = \int_Q \oldnorm{\bz}^2_{\mathbb{A}(\vect{q})} \dd x \dd t.
\end{equation}
Furthermore, noticing that
\begin{equation}\label{dec:222}
   \begin{split}
   \int_Q (\bz, \bz^n - \bz)_{\mathbb{A}(\vect{q}^n)} \dd x \dd t  &=  \int_Q \bz \cdot (\mathbb{A}(\bq^n) - \mathbb{A}(\bq)) (\bz^n - \bz) \dd x \dd t  +
   \int_Q \bz \cdot \mathbb{A}(\bq) (\bz^n - \bz) \dd x \dd t \\ &=: I^n_1 + I^n_2,
   \end{split}
\end{equation}
we see that, as $n\to\infty$,  $I^n_2$ vanishes by virtue of \eqref{dec:12a}. To conclude that $I^n_1$ vanishes as well, we first apply H\"{o}lder's inequality to get that
$$
  |I^n_1| \le \|\bz^n - \bz\|_{L^2(Q;\mathbb{R}^d)} \left(\int_Q |\bz|^2 |\mathbb{A}(\bq^n) - \mathbb{A}(\bq)|^2 \dd x \dd t \right)^{1/2},
$$
and then we notice that $\|\bz^n - \bz\|_{L^2(Q;\mathbb{R}^d)}$ is bounded due to \eqref{dec:12a} and the last integral vanishes again by Lebesgue's dominated convergence theorem. Thus, $\lim_{n\to \infty} (I^n_1 + I^n_2) = 0 $ and the assertion \eqref{dec:11} follows from \eqref{dec:10}-\eqref{dec:222}.
\section{Limit \texorpdfstring{$\varepsilon\to 0_+$}{etozero}}\label{sec:5}

\subsection{The attainment of \texorpdfstring{$\nabla u = \vect{f}(\bq)$}{n} a.e. in \texorpdfstring{$Q$}{Q}} In Sect. \ref{galerkin_section}, assuming that $u_0\in W^{1,\infty}_{per}(\Omega)$ satisfies \eqref{flux_small_data} and $g\in L^2(Q)$, we established, for any $a>0$ and $\varepsilon\in (0,1)$, the existence of unique weak solution to \eqref{zadani_eps} satisfying \eqref{wf-eps}. Furthermore, particularly in Subsect. \ref{apriori_revisited}, we showed that $\{(u^{\varepsilon}, \bq^{\varepsilon})\}_{\varepsilon\in (0,1)}$ satisfies the estimates \eqref{dec:17}, \eqref{dec:17a}, \eqref{dec:14} and \eqref{dec:13b}. As a consequence of these estimates (that are uniform w.r.t. $\varepsilon$) and the Aubin-Lions compactness lemma, one can find $\varepsilon_m\to 0$ and the corresponding sequence $(u^m,\bq^m) := (u^{\varepsilon_m}, \bq^{\varepsilon_m})$ such that
\begin{subequations}\label{hr:1}
\begin{align}
    \label{hr:2}
    u^m&\rightharpoonup u &&\text{~weakly in~} L^2\left(0,T; W^{2,2}_{per}(\Omega)\right),\\ \label{hr:3}
    \partial_t{u}^m&\rightharpoonup \, \partial_t {u} &&\text{~weakly in~} L^{2}\left(0,T;L^{2}(\Omega)\right), \\ \label{hr:4}
    \nabla u^m&\to \nabla u &&\text{~strongly in~} L^2\left(0,T; L^2_{per}(\Omega; \mathbb{R}^d)\right),\\
    \label{hr:5}
    \nabla u^m&\to \nabla u &&\text{~a.e. in~} Q,
\end{align}
\end{subequations}
and also, using \eqref{hr:5} and Egoroff's theorem on one side and \eqref{dec:17a} and Chacon's biting lemma (see \cite{BaMu89}) on the other side, there is a ${\bq\in L^1(Q; \R^d)}$ such that for each $\delta>0$ there exists a $\tilde{Q}_{\delta}\subset Q$ fulfilling $\tilde{Q}_{\delta_2}\subset \tilde{Q}_{\delta_1}$ if $\delta_1\le \delta_2$ as well as  $|Q\setminus \tilde{Q}_{\delta}|\le \delta$ such that
\begin{equation}
\begin{aligned}\label{biting}
 \bq^m&\rightharpoonup \bq  &&\text{weakly~in~}L^{1}(\tilde{Q}_{\delta}; \R^d),\\
 \nabla u^m &\to \nabla u  &&\text{strongly~in~}L^{\infty}(\tilde{Q}_{\delta}; \R^d).
\end{aligned}
\end{equation}
Further, we denote
$$
Q_{\delta}:=\tilde{Q}_{\delta}\cap \left\{(t,x)\in Q; \, |\bq(t,x)|\le \frac{1}{\delta}\right\}
$$
and it follows from \eqref{biting} that
$$
|Q\setminus {Q}_{\delta}|\le |Q\setminus \tilde{Q}_{\delta}|+|\left\{(t,x)\in Q; \, |\bq(t,x)|> \delta^{-1}\right\}|\le \delta\left(1+\int_Q |\bq| \dd x \dd t \right)\le C\delta.
$$
Hence, using the (strict) monotonicity of $\vect{f}$, see Lemma \ref{monotonie}, the facts that $\vect{f}(\bq)\in L^{\infty}(Q;\mathbb{R}^d)$ and $\vect{f}(\bq^m) = \nabla u^m - \varepsilon_m \bq^m$, see \eqref{wf-ce}, the convergence properties \eqref{biting}, the obvious relation $Q_{\delta}\subset \tilde{Q}_{\delta}$,  and the fact that $\bq$ is bounded (depending on $\delta$) on $Q_{\delta}$, we observe that
$$
\begin{aligned}
0&\le \limsup_{m\to \infty}\int_{Q_{\delta}} \bigl(\vect{f}\left(\bq^m\bigr)-\vect{f}\left(\bq\right)\right) \cdot (\bq^m-\bq)\dd x \dd t =\limsup_{m\to \infty}\int_{Q_{\delta}} \vect{f}(\bq^m) \cdot (\bq^m-\bq)\dd x \dd t\\
&=\limsup_{m\to \infty}\int_{Q_{\delta}} \nabla u^m \cdot (\bq^m-\bq) - \varepsilon_m \bq^m\cdot (\bq^m-\bq)  \dd x \dd t \\
&=\limsup_{m\to \infty}\int_{Q_{\delta}}  (\nabla u^m- \nabla u) \cdot (\bq^m - \bq) + \nabla u \cdot (\bq^m-\bq) - \varepsilon_m |\bq^m|^2 + {\varepsilon_m} \bq^m \cdot \bq  \dd x \dd t\\
&\le 0.
\end{aligned}
$$
This implies that there is a subsequence (that we again denote by $\bq^m$) such that
$$
  \lim_{m\to\infty}  \bigl(\vect{f}\left(\bq^m\bigr)-\vect{f}\left(\bq\right)\right) \cdot (\bq^m-\bq) = 0 \quad \textrm{a.e. in } Q_{\delta}.
$$
As $\vect{f}$ is strictly monotone, we conclude (referring for example to Lemma 6 in \cite{DalMasoMurat1998}) that
$$
\bq^m \to \bq \quad \textrm{a.e. in } Q_\delta .
$$
However, as $\delta>0$ is arbitrary and $|Q\setminus Q_\delta|\le C\delta$, this yields
\begin{equation}\label{dec:22}
    \bq^m \to \bq \quad \textrm{a.e. in } Q.
\end{equation}
As $\vect{f}$ is continuous, letting $m\to \infty$ in $\vect{f}(\bq^m) = \nabla u^m - \varepsilon_m \bq^m$ (valid a.e. in $Q$) and using \eqref{hr:5} and \eqref{dec:22}, we conclude that \eqref{constit_old} holds. \\

\subsection{Limit in the governing evolutionary equation}~It remains to show that \eqref{rovnice} holds. Towards this goal,
we ``renormalize" the equation \eqref{wf-rovnice} for $\varepsilon_m$-approximation with the help of smooth, compactly supported approximations of unity denoted by $\tau_k$, which are the functions of $|\bq^m|$. The required equation \eqref{rovnice} is then obtained by a careful study of the limiting process as $m\to\infty$ and $k\to\infty$.

It follows from \eqref{wf-rovnice} that, for all $m\in \mathbb{N}$,
\begin{align}
\label{flux_rce}
    \int_Q \partial_t u^m\,\varphi+ \vect{q}^m\cdot \nabla\varphi\dd x\dd t&=\int_Q g\, \varphi\dd x \dd t &&\text{~for all~}\varphi\in L^2\left(0,T;W^{1,2}_{per}(\Omega)\right).
\end{align}
In order to make use of these relations in the absence of weak convergence of $\vect{q}^m$ in $L^1(Q)$, we consider
$$
\psi\in
L^{\infty}\left(0,T;W_{per}^{1,\infty}(\Omega)\right),
$$
and set as a test function $\varphi$ in \eqref{flux_rce}
\begin{equation}
    \varphi:=\tau_k(\abs{\vect{q}^m}) \psi, \label{dec:23}
\end{equation}
where $\tau_k$, $k\in \mathbb{N}$, ``approximates unity'', i.e. $\tau_k\in C_0^{\infty}\left([0,\infty)\right)$ satisfies for all  $k\in\N$
the following conditions: $0\le \tau_k(s)\le 1$ for all $s\in[0,\infty)$, $\tau_k(s)= 1$ on $[0,k]$, $\tau_k(s)=0$ on $[ k+1, \infty)$ and $-2\leq \tau_k'(s)\leq 0$ for all $s\in (k,k+1)$. Note that, for fixed $m$, the test function specified in \eqref{dec:23} is an admissible test function due to \eqref{dec:13b}.

Inserting \eqref{dec:23} into \eqref{flux_rce}, we obtain
\begin{align}
\label{flux_rce_a}
    \int_Q \partial_t u^m\,\tau_k(\abs{\vect{q}^m}) \psi \dd x\dd t + \int_Q \vect{q}^m \tau_k(\abs{\vect{q}^m}) \cdot \nabla \psi \dd x\dd t&=\int_Q g\, \tau_k(\abs{\vect{q}^m}) \psi \dd x \dd t - \int_Q \vect{q}^m \cdot \nabla \tau_k(\abs{\vect{q}^m}) \psi \dd x\dd t.
\end{align}
Letting $m\to\infty$ and $k\to \infty$ in \eqref{flux_rce_a}, our aim is to show that we can remove $m$ and replace $\tau_k$ by $1$ in the first three integrals, while the last integral vanishes. Each term requires a special treatment.

To treat the term involving the time derivative, we first observe that
\begin{equation}
\begin{split}\label{dec:24}
  I_1^{m,k}:&= \int_Q \left( \partial_tu^m\,\tau_k (|\bq^m|) \psi - \partial_t u \psi \right) \dd x\dd t  = \int_Q \partial_tu^m \left( \tau_k(|\bq^m|) - \tau_k(|\bq|) \right) \psi \dd x\dd t \\
  &+ \int_Q \left( \partial_tu^m - \partial_t u\right) \tau_k(|\bq|)  \psi  \dd x\dd t - \int_Q \partial_t u\, (1 - \tau_k(|\bq|)) \psi \dd x\dd t =: J_1^{m,k} + J_2^{m,k} - J_3^{m,k}.
\end{split}
\end{equation}
By H\"{o}lder's inequality, \eqref{dec:14}, \eqref{dec:22} and Lebesgue's dominated convergence theorem, we observe that
$$
  |J_1^{m,k}| \le \|\partial_t u^m\|_{L^2(Q)} \|\psi\|_{L^{\infty}(Q)} \left( \int_Q \left| \tau_k(|\bq^m|) - \tau_k(|\bq|) \right|^2 \dd x\dd t \right)^{1/2} \to 0 \quad \textrm{ as } m\to\infty.
$$
By \eqref{hr:3}, $J_2^{m,k} \to 0$ as $m\to \infty$. Using Levi's monotone convergence theorem we also get
$$
  |J_3^{m,k}| \le \int_Q |\partial_t u|\,\, |\psi|\left(1 - \tau_k(|\bq|) \right) \dd x\dd t \to 0 \quad \textrm{ as } k\to\infty.
$$
Consequently, it follows from \eqref{dec:24} and the above arguments that
\begin{equation*}
    \lim_{k\to\infty} \lim_{m\to \infty} \int_Q \partial_t u^m\,\tau_k(\abs{\vect{q}^m}) \psi \dd x\dd t = \int_Q \partial_t u \, \psi \dd x\dd t.
\end{equation*}
Even simpler arguments give
\begin{equation}\label{dec:26}
    \lim_{k\to\infty} \lim_{m\to \infty} \int_Q g\,\tau_k(\abs{\vect{q}^m}) \psi \dd x\dd t = \int_Q g \, \psi \dd x\dd t.
\end{equation}
Since, by \eqref{dec:22} and Lebesgue's dominated convergence theorem,
$$
\left| \int_Q \left(\bq^m \tau_k(|\bq^m|) - \bq \tau_k(|\bq|)\right) \cdot \nabla \psi \dd x\dd t \right| \le \|\nabla\psi\|_{L^{\infty}(Q)} \int_Q |\bq^m \tau_k(|\bq^m|) - \bq \tau_k(|\bq|)| \dd x\dd t \to 0 \quad \textrm{as } m\to \infty,
$$
we also observe (again using Levi's monotone convergence theorem) that
\begin{equation*}
    \lim_{k\to\infty} \lim_{m\to \infty} \int_Q \bq^m \tau_k(|\bq^m|) \cdot\nabla \psi \dd x\dd t = \int_Q \bq \cdot \nabla \psi \dd x\dd t.
\end{equation*}
It remains to show that the last term in \eqref{flux_rce_a} tends to zero as $m,k \to \infty$. To prove this, we will incorporate the weighted $L^2$-estimates for $\nabla \bq^m$, see \eqref{dec:13b}. Before starting to treat this term, we introduce, for every $k\in\N$, an auxiliary function $G_k$ through
\begin{equation*}
    G_k(t):=\int_0^t\tau'_k(s)(1+s^a)^{\frac 1a}\dd s\qquad\text{~for ~}t\in[0,\infty)
\end{equation*}
and observe that $G_k(t) =0$ on $[0,k]$ and
\begin{equation}\label{G_odhad}
    \abs{G_k(t)}\leq\int_k^{k+1}\abs{\tau'_k(s)} 2^{\frac1a} s \dd s\leq 2^{\frac{a+1}{a}} (1+k) \leq \mathcal{C}(a) t \quad \textrm{ for every } t\geq k.
\end{equation}
Let us now rewrite the last term in \eqref{flux_rce_a} in the following manner:
\begin{align}\label{flux_hard_1}
\begin{split}
  \int_Q\psi\vect{q}^m\cdot\nabla\tau_k(\abs{\vect{q}^m})\dd x\dd t&=\int_Q\psi\frac{\vect{q}^m\cdot\nabla(\abs{\vect{q}^m})}{(1+\abs{\vect{q}^m}^a)^{\frac 1a}}\tau_k'(\abs{\bq^m})\left(1+\abs{\vect{q}^m}^a\right)^{\frac 1a}\dd x\dd t=\int_Q\psi \, \vect{f}(\bq^m) \cdot\nabla G_k(\abs{\vect{q}^m}) \dd x\dd t\\
    &=-\int_Q \nabla\psi\cdot\vect{f}(\vect{q}^m) G_k(\abs{\vect{q}^m}) \dd x\dd t - \int_Q\psi G_k(\abs{\vect{q}^m}) \di \vect{f}(\bq^m) \, \dd x\dd t =: J_4^{m,k} + J_5^{m,k}.
    \end{split}
\end{align}
To show that $J_4^{m,k}$ vanishes as $m\to \infty$ and $k\to \infty$, we first employ, for any fixed $k$, \eqref{dec:22} and Lebesgue's dominated convergence theorem (noticing that
$\abs{\nabla\psi\cdot\vect{f}(\vect{q}^m)} G_k(\abs{\vect{q}^m}) \leq 2^{\frac{a+1}{a}} (1+k)\Vert\nabla\psi\Vert_{L^{\infty}(Q;\R^d)}$)
and obtain
$$
\lim_{m\to \infty}\int_Q\nabla\psi\cdot\vect{f}(\vect{q}^m) G_k(\abs{\vect{q}^m})\dd x\dd t \to \int_Q \nabla\psi\cdot \vect{f}(\vect{q}) G_k(\abs{\vect{q}})\dd x\dd t.
$$
Since $G_k(t)= 0$ on $[0,k]$, we conclude from the estimate \eqref{G_odhad} and the fact that $\vect{q}\in L^1(Q;\R^d)$ that
\begin{align*}
    \abs{\int_{Q} \nabla\psi\cdot\vect{f}(\vect{q}) G_k(\abs{\vect{q}})\dd x\dd t}&\leq \int_{Q\cap\{\abs{\vect{q}}>k\}}\abs{ \nabla\psi\cdot\vect{f}(\vect{q}) G_k(\abs{\vect{q}})} \dd x\dd t\\&\leq \mathcal{C}(a)\Vert\nabla\psi\Vert_{L^{\infty}(Q;\R^d)}\int_{Q\cap\{\abs{\vect{q}}>k\}}\abs{\vect{q}}\dd x\dd t   \xrightarrow{k\to\infty}0.
\end{align*}
Hence, $\lim_{k\to\infty}\lim_{m\to \infty} J_4^{m,k} = 0$.

In order to show that also the term $J_5^{m,k}$ vanishes as $m\to\infty$ and $k\to \infty$ we need to proceed more carefully. First, recalling \eqref{dec:3}, we notice that
\begin{equation*}
    \di\left(\vect{f}(\bq^m)\right) = (\mathbb{A} (\bq^m))_{ij} \partial_i q^m_j = (\mathbb{A} (\bq^m))_{ij} \partial_s q^m_j \delta_{is} = \sum_{s=1}^d \left(\partial_s \bq^m, \vect{e}_s\right)_{\mathbb{A}(\bq^m)}\qquad\text{~a.e. in~}Q,
\end{equation*}
where $\vect{e}_s\in\R^d$ is the $s$th vector of the canonical basis in $\mathbb{R}^d$, $s=1, \dots, d$.
This allows us to rewrite and estimate $J^{m,k}_5$ introduced in \eqref{flux_hard_1} as follows:
\begin{align*}
\left|J^{m,k}_5\right| &= \left| \int_Q \sum_{s=1}^d \left( \partial_s \bq^m, \psi G_k(|\bq^m|)\vect{e}_s\right)_{\mathbb{A}(\bq^m)} \right| \dd x\dd t \le \|\psi\|_{L^{\infty}(Q)} \sum_{s=1}^d \int_Q  \oldnorm{\partial_s \bq^m}_{\mathbb{A}(\vect{q}^m)} \oldnorm{G_k(|\bq^m|)\vect{e}_s}_{\mathbb{A}(\vect{q}^m)} \dd x\dd t \\
&\le \|\psi\|_{L^{\infty}(Q)} \sum_{s=1}^d \left(\int_Q  \oldnorm{\nabla \bq^m}^2_{\mathbb{A}(\vect{q}^m)} \dd x\dd t\right)^{\frac12} \left(\int_Q  \oldnorm{G_k(|\bq^m|)\vect{e}_s}^2_{\mathbb{A}(\vect{q}^m)} \dd x\dd t \right)^{\frac12} \\
&\!\!\overset{\eqref{dec:13b}}{\le} \mathcal{C}\!\left(\Abs{g}_{L^2(Q)}\!,\Abs{u_0}_{W^{1,2}_{per}(\Omega)}\right) \|\psi\|_{L^{\infty}(Q)} \sum_{s=1}^d  \left(\int_Q  \oldnorm{G_k(|\bq^m|)\vect{e}_s}^2_{\mathbb{A}(\vect{q}^m)} \dd x\dd t \right)^{\frac12} \\
& \!\!\overset{\eqref{odhady_s_carkou_new}}{\le} \mathcal{C}\!\left(d,\Abs{g}_{L^2(Q)}\!,\Abs{u_0}_{W^{1,2}_{per}(\Omega)}\right) \|\psi\|_{L^{\infty}(Q)} \left(\int_Q  \frac{|G_k(|\bq^m|)|^2}{(1+|\vect{q}^m|^a)^{1/a}} \dd x\dd t \right)^{\frac12}.
\end{align*}
Letting $m\to\infty$ in the last term, using \eqref{dec:22}, \eqref{G_odhad} and Lebesgue's dominated convergence theorem, we get
\begin{equation}\label{dec:33}
    \limsup_{m\to\infty} \, \left|J^{m,k}_5\right| \le \mathcal{C}\!\left(d,\Abs{g}_{L^2(Q)}\!,\Abs{u_0}_{W^{1,2}_{per}(\Omega)}\right) \|\psi\|_{L^{\infty}(Q)} \left(\int_Q  \frac{|G_k(|\bq|)|^2}{(1+|\vect{q}|^a)^{1/a}} \dd x\dd t \right)^{\frac12}.
\end{equation}
However, as $G_k(s)= 0$ on $[0,k]$ and \eqref{G_odhad} holds, we further observe that
\begin{align*}
    \int_{Q} \frac{G_k^2(\abs{\vect{q}})}{(1+\abs{\vect{q}}^a)^{1/a}}\dd x \dd t &=\int_{Q\cap\{\abs{\vect{q}}>k\}} \frac{G_k^2(\abs{\vect{q}})}{(1+\abs{\vect{q}}^a)^{1/a}} \dd x \dd t \\
    &\le \mathcal{C}(a) \int_{Q\cap\{\abs{\vect{q}}>k\}} \frac{\abs{\vect{q}}^2}{1+\abs{\vect{q}}} \dd x \dd t
    \le \mathcal{C}(a) \int_{Q\cap\{\abs{\vect{q}}>k\}} \abs{\vect{q}} \dd x \dd t \to 0 \quad \textrm{ as } k\to \infty.
\end{align*}
Hence $\lim_{k\to\infty}\lim_{m\to\infty} \left|J_5^{m,k}\right| = 0$ and, taking all computations starting at \eqref{flux_hard_1} into consideration, the last term in \eqref{flux_rce_a} vanishes as $m\to\infty$ and $k\to\infty$. The proof of the first part of Theorem \ref{main_thm} is complete.

\section{Improved time derivative estimates and higher integrability of the flux for \texorpdfstring{$a\in (0, 2/(d+1))$}{a}}

In order to prove the second part of Theorem \ref{main_thm}, we will combine the uniform spatial derivative estimates established in \eqref{dec:13b} for $(u^\varepsilon,\bq^\varepsilon)$ together with the uniform time derivative estimates that we are going to prove next.

\subsection{Improved time derivative estimates}\label{time_derivative}~
Consider, for any $\varepsilon\in(0,1)$, the unique weak solution $(u^\varepsilon,\bq^\varepsilon)$ to \eqref{zadani_eps} satisfying \eqref{dec:34} and \eqref{wf-eps}. It follows from \eqref{wf-rovnice}, \eqref{wf-in} and the assumption $u_0\in W^{1,2}_{per}(\Omega)$ that, for $\tau\in(0,T]$,
\begin{equation}\label{dec:54}
    \int_{\Omega} (u^\varepsilon(\tau,\cdot) - u_0) u_0 \dd x  + \int_0^{\tau} \int_{\Omega} \bq^{\varepsilon}\cdot\nabla u_0 \dd x \dd s = \int_0^{\tau} \int_{\Omega} g\, u_0 \dd x \dd s.
\end{equation}
By setting $\varphi = u^\varepsilon$ in \eqref{wf-rovnice}, followed by integration over time between $0$ and $\tau$, we also have
\begin{equation}\label{dec:50}
    \|u^\varepsilon(\tau,\cdot)\|_{L^2(\Omega)}^2 - \|u_0\|_{L^2(\Omega)}^2 + 2 \int_0^{\tau} \int_{\Omega} \bq^{\varepsilon}\cdot\nabla u^\varepsilon \dd x \dd s = 2 \int_0^{\tau} \int_{\Omega} g\, u^\varepsilon \dd x \dd s.
\end{equation}

\noindent\emph{Step 1.} For any $z:[0,T]\times\Omega \to \R$ and for $\tau\in\R$ such that $t+\tau\in[0,T]$, we set
$$
  \delta_{\tau} z(t,x) := \frac{z(t+\tau,x) - z(t,x)}{\tau}.
$$
Taking the weak formulation \eqref{wf-rovnice} at $t+\tau$, followed by subtracting \eqref{wf-rovnice} at $t$, and taking then $\varphi = \frac{1}{\tau} \delta_{\tau}u^\varepsilon$ as a test function in the resulting equation, we obtain
\begin{equation}\label{dec:51}
  \frac12 \ddt \|\delta_{\tau} u^\varepsilon\|_{L^2(\Omega)}^2 + \int_{\Omega} \delta_{\tau}\bq^{\varepsilon} \cdot \nabla \delta_{\tau} u^\varepsilon \dd x = \int_{\Omega} \delta_{\tau} g \, \delta_{\tau} u^\varepsilon \dd x.
\end{equation}
Using \eqref{wf-ce} (or \eqref{constit_eps}) and \eqref{dec:3}, we observe that
\begin{align}\label{dec:52}
    \delta_{\tau}\bq^{\varepsilon} \cdot \nabla \delta_{\tau} u^\varepsilon =\delta_{\tau}\bq^{\varepsilon} \cdot \delta_{\tau} \nabla u^\varepsilon = \int_0^1 \delta_{\tau}\bq^{\varepsilon} \cdot \mathbb{A}({\bq}^{\varepsilon}_{\theta,\tau}) \delta_{\tau}\bq^{\varepsilon} \,\dd \theta + \varepsilon |\delta_{\tau}\bq^{\varepsilon}|^2 =\int_0^1 \oldnorm{\delta_{\tau}\bq^{\varepsilon}}^2_{\mathbb{A}({\bq}^{\varepsilon}_{{\theta,\tau}})} \, \dd \theta + \varepsilon |\delta_{\tau}\bq^{\varepsilon}|^2,
\end{align}
where
$$
{\bq}^{\varepsilon}_{{\theta,\tau}}(t,x):= {\bq}^{\varepsilon}(t,x) + \theta \left( {\bq}^{\varepsilon}(t+\tau,x) - {\bq}^{\varepsilon}(t,x)\right) \quad\textrm{for } \theta\in (0,1).
$$
Inserting \eqref{dec:52} into \eqref{dec:51} and using the Cauchy-Schwarz inequality to estimate the right-hand side and the Gronwall lemma, we conclude that for a.a. $t\in (0,T]$ the following estimates holds:
\begin{equation}\label{dec:53}
  \|\delta_{\tau} u^\varepsilon(t,\cdot)\|_{L^2(\Omega)}^2 + \int_0^t \int_{\Omega} \int_0^1 \oldnorm{\delta_{\tau}\bq^{\varepsilon}}^2_{\mathbb{A}({\bq}^{\varepsilon}_{\theta,\tau})} \, \dd \theta + \varepsilon |\delta_{\tau}\bq^{\varepsilon}|^2 \dd x \dd s \le C \, e^T \left(\|\delta_{\tau} u^\varepsilon(0,\cdot)\|_{L^2(\Omega)}^2 + \int_0^T \|\delta_{\tau} g \|_{L^2(\Omega)}^2 \dd s\right).
\end{equation}
This would lead to the required $\varepsilon$-independent estimates provided that we can control $\|\delta_{\tau} u^\varepsilon(0,\cdot)\|_{L^2(\Omega)}^2$ uniformly w.r.t. $\varepsilon$.

\noindent\emph{Step 2.} Towards this aim, we start by noticing that trivially
$$
  \|\delta_{\tau} u^\varepsilon(0,\cdot)\|_{L^2(\Omega)}^2 = \frac{1}{\tau^2} \|u^\varepsilon(\tau,\cdot) - u_0\|_{L^2(\Omega)}^2
$$
and
\begin{equation}\label{dec:55}
\|u^\varepsilon(\tau,\cdot) - u_0\|_{L^2(\Omega)}^2 = \|u^\varepsilon(\tau,\cdot)\|_{L^2(\Omega)}^2 - \|u_0\|_{L^2(\Omega)}^2
-2\int_{\Omega} (u^{\varepsilon}(\tau, \cdot) - u_0)\, u_0 \dd x.
\end{equation}
Inserting \eqref{dec:54} and \eqref{dec:50} into \eqref{dec:55} we get
\begin{equation*}
\begin{split}
\frac12 \|u^\varepsilon(\tau,\cdot) - u_0\|_{L^2(\Omega)}^2 &=
\int_0^\tau\int_{\Omega} g u^\varepsilon \dd x\dd s
-  \int_0^\tau\int_{\Omega} \bq^\varepsilon\cdot\nabla u^\varepsilon \dd x\dd s \\ &- \int_0^\tau \int_\Omega g u_0 \dd x \dd s + \int_0^\tau \int_\Omega \bq^\varepsilon \cdot \nabla u_0 \dd x \dd s.
\end{split}
\end{equation*}
This can be rewritten as
\begin{equation}\label{dec:57}
\begin{split}
\frac12 \|u^\varepsilon(\tau,\cdot)  - u_0\|_{L^2(\Omega)}^2
&+\int_0^\tau\int_{\Omega} (\bq^\varepsilon - \bq^\varepsilon(0, \cdot)) \cdot(\nabla u^\varepsilon - \nabla u_0) \dd x\dd s \\
&= \int_0^\tau\int_{\Omega} (g - g(0,\cdot) (u^\varepsilon - u_0) \dd x\dd s + \int_0^\tau\int_{\Omega} g(0,\cdot) (u^\varepsilon - u_0) \dd x\dd s \\ &+ \int_0^\tau \int_{\Omega} \di \bq^\varepsilon(0,\cdot) \left( u^\varepsilon - u_0\right)  \dd x \dd s,
\end{split}
\end{equation}
where $\bq^\varepsilon(0,\cdot)$ is defined, in accordance with Subsect. \ref{apriori_revisited}, through
\begin{equation}\label{dec:57a}
    \nabla u_0 = \vect{f}^{\varepsilon}(\bq^\varepsilon(0,\cdot)) = \frac{\bq^\varepsilon(0,\cdot)}{(1+|\bq^\varepsilon(0,\cdot)|^a)^{1/a}} + \varepsilon \bq^\varepsilon(0,\cdot).
\end{equation}
Since $\nabla u^\varepsilon = \vect{f}^{\varepsilon}(\bq^\varepsilon)$ and $\vect{f}^{\varepsilon}$ is monotone, the second term at the left-hand side of \eqref{dec:57} is nonnegative. Introducing the notation
$$
  y(\tau):=\frac12 \int_0^\tau \|u^\varepsilon(s,\cdot)  - u_0\|_{L^2(\Omega)}^2 \dd s
$$
and
$$
  A(s):= \|g(s,\cdot) - g(0,\cdot)\|_{L^2(\Omega)}
  + \|g(0,\cdot)\|_{L^2(\Omega)} + \|\nabla \bq^{\varepsilon}(0,\cdot)\|_{L^2(\Omega;\R^{d\times d})},
$$
we conclude from \eqref{dec:57}, using H\" older's inequality, that
\begin{equation}\label{Jan:0}
    y'(\tau) = \frac12 \|u^\varepsilon(\tau,\cdot)  - u_0\|_{L^2(\Omega)}^2 \dd s \le  \int_0^{\tau} A(s) \|u^{\varepsilon}(s,\cdot) - u_0\|_{L^2(\Omega)} \dd s \le \left( \int_0^\tau A^2(s) \, \dd s\right)^{1/2} 2 y(\tau)^{1/2}.
\end{equation}
This (together with relabelling $s$ on $v$ and $\tau$ on $s$) implies that
\begin{equation} \label{Jan:1a}
    (y^{1/2})'(s) \le \left( \int_0^s A^2(v) \, \dd v\right)^{1/2}.
\end{equation}
Since $y(0)=0$, integrating \eqref{Jan:1a} over $(0,\tau)$ and using then H\"{o}lder's inequality, we get
\begin{equation*}
    y(\tau) \le \left( \int_0^\tau \left( \int_0^s A^2(v) \, \dd v\right)^{1/2} \dd s\right)^{2} \le \tau \int_0^\tau \int_0^s A^2(v) \, \dd v \dd s \le \tau^2 \int_0^\tau A^2(s) \, \dd s,
\end{equation*}
which implies that
\begin{equation*}
    (y(\tau))^{1/2} \le \tau \left( \int_0^\tau A^2(s) \dd s \right)^{1/2}.
\end{equation*}
Using this to estimate the last term in \eqref{Jan:0}, it follows from \eqref{Jan:0} that
\begin{equation}\label{Jan:1}
    \|u^\varepsilon(\tau,\cdot)  - u_0\|_{L^2(\Omega)}^2 \dd s \le 4\tau \int_0^\tau A^2(s) \, \dd s.
\end{equation}
Recalling the definition of $A$, \eqref{Jan:1} leads to  (using also $1/\tau^2 \le 1/s^2$)
\begin{equation*}
    \begin{split}
    \|u^\varepsilon(\tau,\cdot)  - u_0\|_{L^2(\Omega)}^2 &\le C \tau \int_0^\tau \|g(s,\cdot) - g(0,\cdot)\|^2_{L^2(\Omega)}
  + \|g(0,\cdot)\|^2_{L^2(\Omega)} + \|\nabla \bq^{\varepsilon}(0,\cdot)\|^2_{L^2(\Omega;\R^{d\times d})} \, \dd s \\
    &\le C \tau^2 \left (\|\nabla \bq^{\varepsilon}(0,\cdot)\|^2_{L^2(\Omega;\R^{d\times d})} + \|g(0,\cdot)\|^2_{L^2(\Omega)} + \tau \int_0^\tau \|\delta_{s} g(0,\cdot)\|_{L^2(\Omega)}^2 \dd s \right) .
    \end{split}
\end{equation*}
This finally gives
\begin{equation*}
\|\delta_{\tau} u^\varepsilon(0,\cdot)\|_{L^2(\Omega)}^2 \le
C\left( \|\nabla \bq^{\varepsilon}(0,\cdot)\|^2_{L^2(\Omega;\R^{d\times d})} + \|g(0,\cdot)\|^2_{L^2(\Omega)} + \tau \int_0^\tau \|\delta_{s} g(0,\cdot)\|_{L^2(\Omega)}^2 \dd s\right).
\end{equation*}
As $g\in W^{1,2}\left(0,T;L^2(\Omega)\right)$ and $W^{1,2}\left(0,T;L^2(\Omega)\right)\hookrightarrow C([0,T];L^2(\Omega))$, the second and the third terms on the right-hand side are bounded.\footnote{Note that it would be sufficient to assume that $g\in W^{\beta,2}\left(0,T;L^2(\Omega)\right)$ for some $\beta>1/2$.} Hence, we finally get
\begin{equation}\label{dec:58}
\|\delta_{\tau} u^\varepsilon(0,\cdot)\|_{L^2(\Omega)} \le \mathcal{C}\left( \|g\|_{W^{1,2}(0,T;L^2(\Omega))}\right) + C \|\nabla \bq^{\varepsilon}(0,\cdot)\|_{L^2(\Omega;\R^{d\times d})}.
\end{equation}

In order to estimate $\|\nabla \bq^{\varepsilon}(0,\cdot)\|_{L^2(\Omega;\R^{d\times d})}$, we first recall that it follows from \eqref{flux_small_data} and \eqref{dec:57a} that
$$
U\ge |\nabla u_0| = \left(\frac{1}{(1+|\bq^{\varepsilon}(0,\cdot)|^a)^{\frac{1}{a}}}+ \varepsilon\right) |\bq^\varepsilon(0,\cdot)|\ge \frac{|\bq^{\varepsilon}(0,\cdot)|}{(1+|\bq^{\varepsilon}(0,\cdot)|^a)^{\frac{1}{a}}} \quad \textrm{ a.e. in } Q,
$$
which implies that
\begin{equation}\label{dec:59}
   |\bq^{\varepsilon}(0,\cdot)|^a \le \frac{U^a}{1-U^a} \quad \textrm{ and } \quad (1+|\bq^{\varepsilon}(0,\cdot)|^a)^{1+\frac{1}{a}} \le \frac{1}{(1-U^a)^{1+\frac{1}{a}}}.
\end{equation}
Next, applying the partial derivative w.r.t. $x_j$ to \eqref{dec:57a} and using also \eqref{dec:3} we get
$$
  \nabla \partial_j u_0 = \mathbb{A}(\bq^\varepsilon(0,\cdot)) \partial_j\bq^\varepsilon(0,\cdot) + \varepsilon \partial_j \bq^\varepsilon(0,\cdot).
$$
Taking the scalar product of this identity with $\partial_j \bq^\varepsilon(0,\cdot)$ and summing the result over $j$, $j=1,\dots,d$, we arrive at
$$
  \varepsilon |\nabla \bq^\varepsilon(0,\cdot)|^2 + \oldnorm{\nabla \bq^\varepsilon(0,\cdot)}^2_{\mathbb{A}(\bq^\varepsilon(0,\cdot))} =  \nabla^2 u_0 \cdot \nabla \bq^\varepsilon(0,\cdot) \le |\nabla^2 u_0| \, |\nabla \bq^\varepsilon(0,\cdot)|.
$$
By virtue of \eqref{odhady_s_carkou_new}, this leads to
$$
  \frac{|\nabla \bq^\varepsilon(0,\cdot)|^2}{(1+|\bq^{\varepsilon}(0,\cdot)|^a)^{1+\frac{1}{a}}} \le |\nabla^2 u_0| \, |\nabla\bq^\varepsilon(0,\cdot)| \,\,\implies\,\, |\nabla \bq^\varepsilon(0,\cdot)| \le  |\nabla^2 u_0| (1+|\bq^{\varepsilon}(0,\cdot)|^a)^{1+\frac{1}{a}} \overset{\eqref{dec:59}}{\le} |\nabla^2 u_0|\frac{1}{(1-U^a)^{1+\frac{1}{a}}},
$$
which implies that
$$
  \|\nabla \bq^{\varepsilon}(0,\cdot)\|_{L^2(\Omega;\R^{d\times d})} \le
  \frac{1}{(1-U^a)^{1+\frac{1}{a}}} \|\nabla^2 u_0\|_{L^2(\Omega;\R^{d\times d})}.
$$
Consequently, using \eqref{dec:53} and \eqref{dec:58}, we conclude that
\begin{equation}\label{dec:60}
  \|\delta_{\tau} u^\varepsilon(t,\cdot)\|_{L^2(\Omega)}^2 + \int_0^t \int_{\Omega} \int_0^1 \oldnorm{\delta_{\tau}\bq^{\varepsilon}}^2_{\mathbb{A}({\bq}^{\varepsilon}_{\theta,\tau})} \, \dd \theta + \varepsilon |\delta_{\tau}\bq^{\varepsilon}|^2 \dd x \dd s \le \mathcal{C}\left( \|g\|_{W^{1,2}(0,T;L^2(\Omega))}, \|u_0\|_{W^{2,2}(\Omega)}, a, U \right).
  \end{equation}

\noindent\emph{Step 3.} Letting $\tau \to 0$ in \eqref{dec:60} ($\varepsilon\in (0,1)$ being fixed) we claim that
\begin{equation}\label{dec:60a}
  \|\partial_t u^\varepsilon(t,\cdot)\|_{L^2(\Omega)}^2 + \int_0^t \int_{\Omega} \oldnorm{\partial_t\bq^{\varepsilon}}^2_{\mathbb{A}({\bq}^{\varepsilon})} + \varepsilon |\partial_t\bq^{\varepsilon}|^2 \dd x \dd s \le \mathcal{C}\left( \|g\|_{W^{1,2}(0,T;L^2(\Omega))}, \|u_0\|_{W^{2,2}(\Omega)}, a, U \right).
\end{equation}
While the limits in the first and third terms of \eqref{dec:60} are standard and are based on weak lower semicontinuity of the $L^2$-norm, the limit in the second term follows from the facts that, as $\tau \to 0$,
\begin{align*}
    \bq^{\varepsilon}_{{\theta,\tau}} &\to \bq^{\varepsilon} &&\textrm{ a.e. in } Q,\\
  \delta_{\tau}\bq^{\varepsilon} &\rightharpoonup \partial_t \bq^\varepsilon &&\textrm{ weakly in } L^2(Q; \R^{d}),
\end{align*}
followed by the convergence arguments established in Subsect. \ref{dec:63}. Thus, \eqref{dec:60a} holds. Consequently, we conclude that $\partial_t u \in L^\infty(0,T; L^2(\Omega))$, which is the first statement of part (ii) of Theorem \ref{main_thm}.

\subsection{Higher integrability result}\label{higher_integrability_section}

It follows from \eqref{dec:13b} and \eqref{dec:60a} that
$$
\int_{Q} \oldnorm{\partial_t\bq^{\varepsilon}}^2_{\mathbb{A}({\bq}^{\varepsilon})} + \oldnorm{\nabla\bq^{\varepsilon}}^2_{\mathbb{A}({\bq}^{\varepsilon})}\le \mathcal{C}\left( \|g\|_{W^{1,2}(0,T;L^2(\Omega))}, \|u_0\|_{W^{2,2}(\Omega)}, U \right)=:\mathcal{C}^*.
$$
Introducing the time-spatial gradient $\nabla_{t,x}u := (\partial_t u, \partial_j u, \dots, \partial_d u)$, we can rewrite the last estimate as
$$
\int_{Q} \oldnorm{\nabla_{t,x} \bq^{\varepsilon}}^2_{\mathbb{A}({\bq}^{\varepsilon})}\le \mathcal{C}^*.
$$
Using the last inequality in \eqref{odhady_s_carkou_new}, it implies that\footnote{$\mathcal{C}^*$ is a generic constant, whose value can change from line to line.}
$$
\int_{Q} \frac{|\nabla_{t,x}\bq^\varepsilon|^2}{(1+|\bq^\varepsilon|)^{a+1}}\dd x \dd t \le \mathcal{C}^*,
$$
and by simple manipulation also
$$
\int_{Q} |\nabla_{t,x} (1+|\bq^\varepsilon|)^{\frac{1-a}{2}}|^2 \dd x \dd t \le \mathcal{C}^*.
$$
Hence, using also \eqref{dec:17a}, we conclude that, for $a\in (0,1)$,
$$
\|(1+|\bq^\varepsilon|)^{\frac{1-a}{2}}\|_{W^{1,2}(Q)}\le \mathcal{C}^*,
$$
and it then follows from Sobolev embedding that
$$
\|(1+|\bq^\varepsilon|)^{\frac{1-a}{2}}\|_{L^p(Q)}\le \mathcal{C}^*,
$$
where $p<\infty$ is arbitrary if $d=1$ and $p=\frac{2(d+1)}{d-1}$ if $d>1$. Thus if $d=1$ and $a<1$ we have a bound in any Lebesgue space. In the case of $d>1$, the above computation gives that
$$
\int_Q (1+|\bq^\varepsilon|)^{\frac{(1-a)(d+1)}{d-1}}\dd x \dd t\le \mathcal{C}^*,
$$
which improves the integrability of  $\{\bq^\varepsilon\}$, uniformly w.r.t. $\varepsilon$, provided that
$$
\frac{(1-a)(d+1)}{d-1}>1 \quad \Leftrightarrow \quad \frac{2}{d+1}>a.
$$
For $\varepsilon_m \to 0$, this piece of information is preserved. Thus, the second assertion of Theorem \ref{main_thm} is established.

\section{Generalization to systems of nonlinear parabolic equations}
\label{rem-final}

Finally, we generalize our problem and formulate the existence and uniqueness results for such a generalization. A detailed proof is not provided as it follows from the combination of the arguments developed in the proof of Theorem \ref{main_thm} above and from the arguments used when proving the result established in \cite{BeBuMa18}, where the stationary case is treated. 

\begin{thm}\label{main_thm2}
Let $\Omega$, $Q$ be as before and let $F:\mathbb{R} \to \R_+$ be a strictly convex  ${C}^{1,1}$ function fulfilling $F(0)=0$. Assume in addition that there exists a positive constant $C$ such that for all $s\in \R$ there holds
$$
C^{-1}|s|-C\le F(|s|) \le C(1+|s|).
$$
For $N\in \mathbb{N}$ arbitrary, set
$$
\vect{f}(\bq):= \partial_{\bq}F(|\bq|), \quad \textrm{ where } \bq\in \R^{d\times N}.
$$
Let $g\in L^2\pigl(0,T; L^{2}\left(\Omega;\R^N\right)\pigr)$, $u_0\in W^{1,2}_{per}\left(\Omega;\R^N\right)$ and there exist a compact set $K\subset \R^{d\times N}$ such that
$$
\nabla u_0(x) \in \vect{f}(K) \quad \textrm{ for a.a. } x\in \Omega.
$$
Then, there exists a unique couple $(u,\bq)$ such that
$$
\begin{aligned}
u&\in W^{1,2}\pigl(0,T; L^2\left(\Omega;\R^N\right)\pigr)\cap L^2\pigl(0,T; W^{2,2}_{per}\left(\Omega;\R^N\right)\pigr),\\
\bq&\in L^1\pigl(0,T; L^1\left(\Omega; \mathbb{R}^{d\times N}\right)\pigr)
\end{aligned}
$$
and
\begin{subequations}\label{flux-integrable_weak_formulationAA}
\begin{align}
       \int_{\Omega} \partial_t u \cdot \varphi + \bq \cdot \nabla \varphi \dd x&=\int_{\Omega} g \cdot \varphi\dd x &&\textrm{for all $\varphi\in W^{1,\infty}_{per}(\Omega;\R^N)$ and a.a. $t\in (0,T)$},\label{rovniceAA}
    \\
    \nabla u&=\vect{f}(\bq) &&\textrm{a.e. in }Q,\label{constit_oldAA}
    \\
     \Vert u(t,\cdot)-u_0\Vert_{L^2(\Omega;\R^N)}&\xrightarrow{t\to 0^+}0. \label{initialAA}
\end{align}
\end{subequations}
\end{thm}


\end{document}